\documentclass{article}

\usepackage{pgf,tikz}
\usetikzlibrary{arrows}
\usepackage{amsmath,amsfonts}
\usepackage{wrapfig}
\usepackage[english]{babel}
\usepackage{subfig}
\usepackage{amsthm}
\usepackage{amssymb}
\usepackage{fancyhdr}
\usepackage{stmaryrd}
\usepackage{graphicx}
\usepackage{setspace}
\usepackage{hyperref}
\usepackage{graphics}
\usepackage{verbatim}
\usepackage[letterpaper,left=3cm,right=3cm,top=3cm,bottom=3cm]{geometry}

\theoremstyle{definition}

\theoremstyle{plain}
\newtheorem{Def}{Definition}[subsection]
\newtheorem{Lemma}[Def]{Lemma}
\newtheorem{Prop}[Def]{Proposition}
\newtheorem{Theo}[Def]{Theorem}

\newtheorem{Remark}[Def]{Remark}
\newtheorem{Theo1}{Theorem}

\title{Billiards in near rectangles}
\date{\today}
\author{Yilong Yang, Haibin Chang}

\begin{document}
\maketitle
\numberwithin{equation}{subsection}
\section{Introduction}
\onehalfspacing
Inside a polygonal billiard table, a billiard ball travels in straight line until it hits an edge. The billiard ball bounces off an edge obeying the law of reflection, i.e. the angle of incidence equals the angle of reflection. See Figure~\ref{fig:Law of Reflection}.

\begin{wrapfigure}{r}{4.5cm}
\vspace{-40pt}
  \begin{center}
    \includegraphics[width=4.5cm]{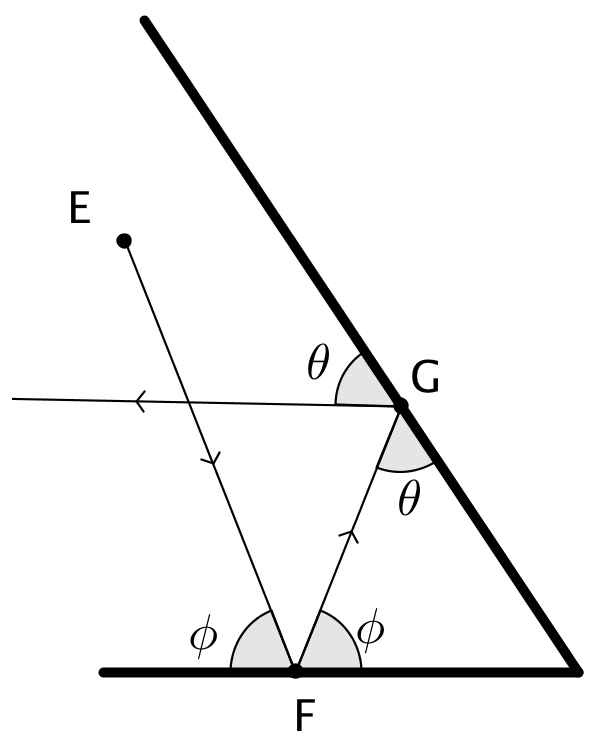}
  \end{center}
  \caption{Law of Reflection}
  \vspace{-10pt}
\label{fig:Law of Reflection}
\vspace{-10pt}
\end{wrapfigure}

The trajectory of the billiard ball is called a \textit{billiard path}. The path is \textit{periodic} if the ball repeats the same path over and over again. We ignore cases where the billiard ball hits one of the vertices of the billiard table.

One major question in this subject is the following: Does every polygon admit a periodic billiard path? It is known that all polygons whose angles are rational multiple of $\pi$ have a periodic billiard path, see \cite{Rational1} and \cite{Rational2}. Schwartz has shown that all triangles with all angles smaller than 100 degree have periodic billiard path \cite{Triangle1}, \cite{Triangle2}. In another paper by Hooper and Schwartz, it is shown that all triangles near enough to isosceles triangles have periodic billiard paths \cite{PatSchwartz}.

In this paper, we consider quadrilaterals which are close to being rectangles. The space $\mathcal{Q}$ of quadrilaterals (modulo similarity) can be ``considered'' as a subset of the $\mathbb{R}^4$, see Section ~\ref{Section:Quadrilateral Space}. The main result of this paper is the following:

\begin{Theo1}
For every rectangle $r \in \mathcal{Q}$, there exists an open neighborhood $U \subset \mathcal{Q}$ containing $r$, such that every quadrilateral in $U$ has a periodic billiard path.
\label{Thrm:1}
\end{Theo1}

Theorem ~\ref{Thrm:epsilon107}  will provide an explicit neighborhood $U$ for the special case when $r$ is a square. 

\begin{wrapfigure}{r}{6cm}
\vspace{-60pt}
  \begin{center}
    \includegraphics[width=6cm]{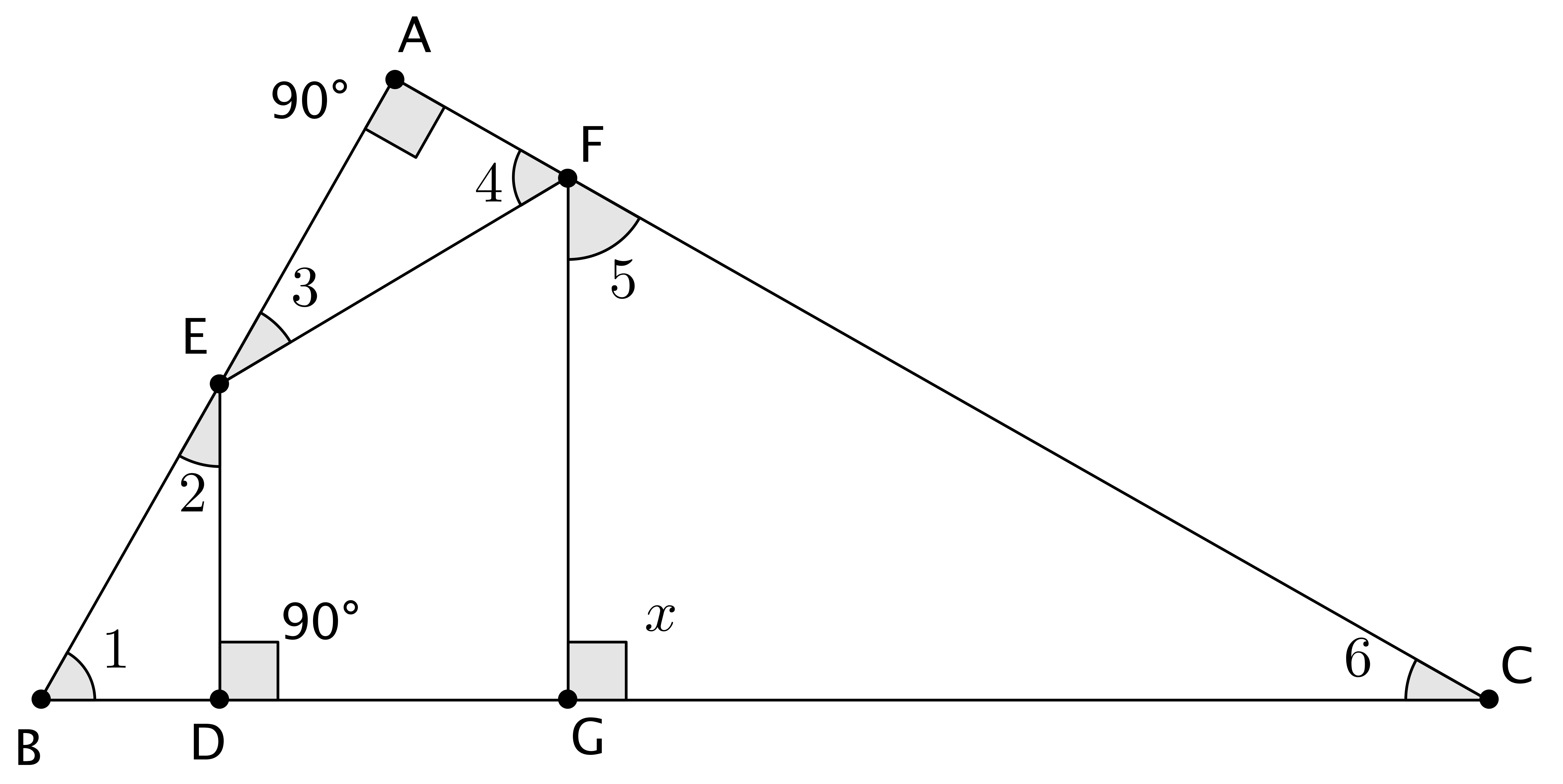}
  \end{center}
  \vspace{-10pt}
  \caption{Right Triangle}
\label{fig:Right Triangle}
\vspace{-10pt}
\end{wrapfigure}
\section{Preliminary}
\subsection{Tools for Studying Billiards}
Consider a right triangle in Figure ~\ref{fig:Right Triangle}. Suppose the billiard ball starts from point $D$ moving upwards, perpendicular to side $BC$. Then it will hit the point $E$, reflect to hit point $F$, and then reflect to hit point $G$. We will leave it to the reader to show that $\angle x = \frac{\pi}{2}$.

The implication is that: after the ball hit $G$, it would bounce back along the exact same trajectory, and continues periodically. So, all right triangles have this periodic billiard path.

\begin{wrapfigure}{r}{5cm}
\vspace{-50pt}
  \begin{center}
    \includegraphics[width=5cm]{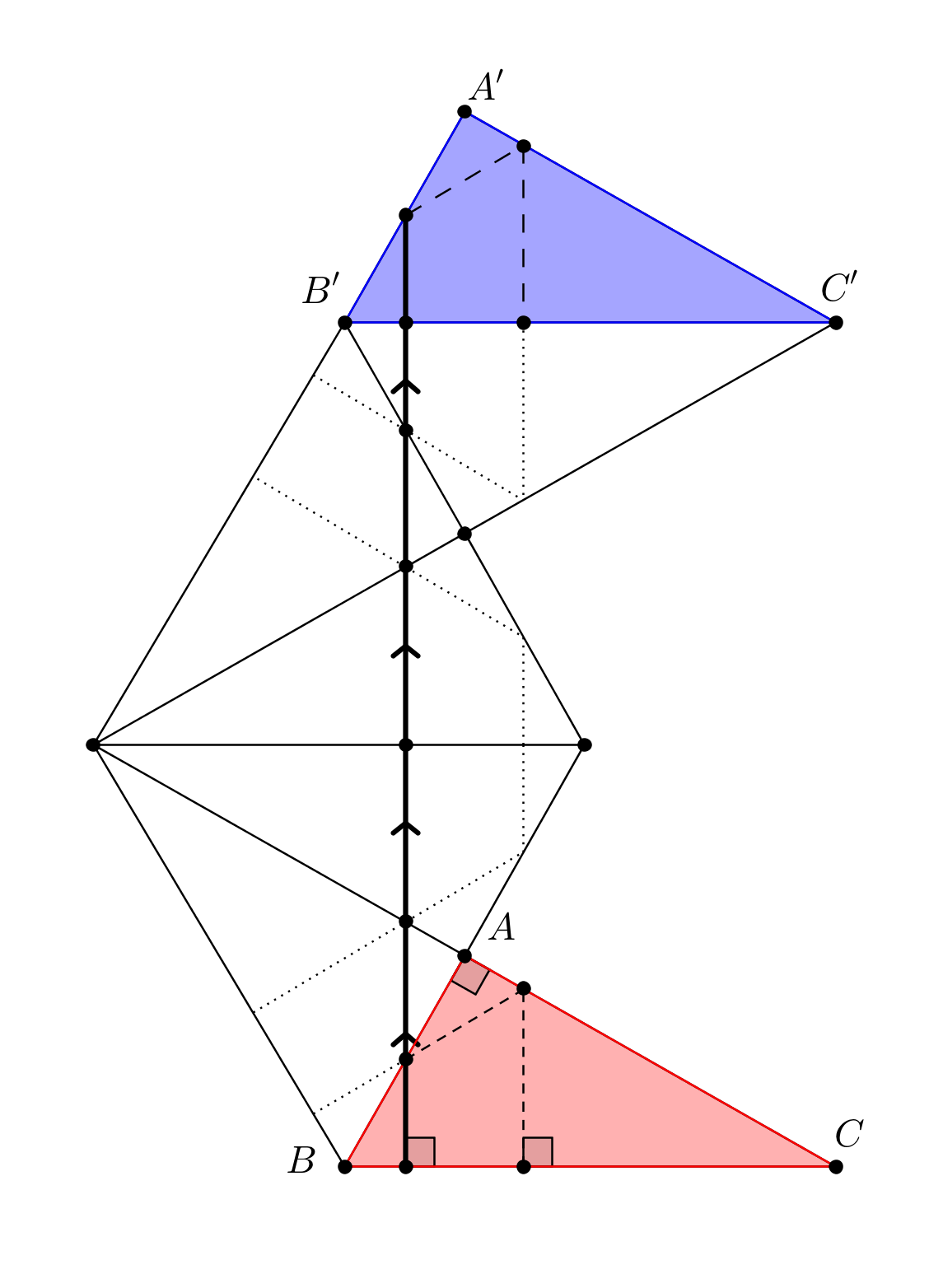}
  \end{center}
\vspace{-20pt}
  \caption{Unfolding for right triangle}
  \vspace{-20pt}
\label{fig:Right Triangle Unfolding}
\end{wrapfigure}

If we label the sides $CB, BA, AC$ using $0,1,2$ respectively, then the sequence $012021$ represents the order in which the trajectory hits the sides of the polygon. We call this finite sequence the $\textit{orbit type}$ of the periodic path, and the length of the orbit type (in the example above, 6) the \textit{combinatorial length}. 

A very useful tool to study billiard paths is the unfolding. See \cite{Tabachnikov} also.

\begin{Def}
Given any orbit type $W=w_1 w_2 w_3\ldots w_n$, and any polygon $P$, the corresponding \textbf{unfolding} is a sequence of polygons $U(W,P)=P_0 P_1 P_2 \ldots P_n$, such that  $P_0=P$, and each $P_j$ is obtained by reflecting $P_{j-1}$ along the edge $w_j$ for $j\geq 1$. 
\label{Def:2.1.3}
\end{Def}

For example, in Figure ~\ref{fig:Right Triangle Unfolding}, each time the ball hits an edge of the polygon, instead of reflecting the billiard path, we reflect the polygon and keep the path straight. The straight line in the unfolding will ``correspond'' to the original periodic path. Note that the two shaded triangles in Figure ~\ref{fig:Right Triangle Unfolding} are related by a translation along the direction of the billiard trajectory. Also, this translation requires a number of 6 reflection, which is exactly the combinatorial length of the orbit. 

\begin{wrapfigure}{l}{4cm}
\vspace{-20pt}
  \begin{center}
    \includegraphics[width=4cm]{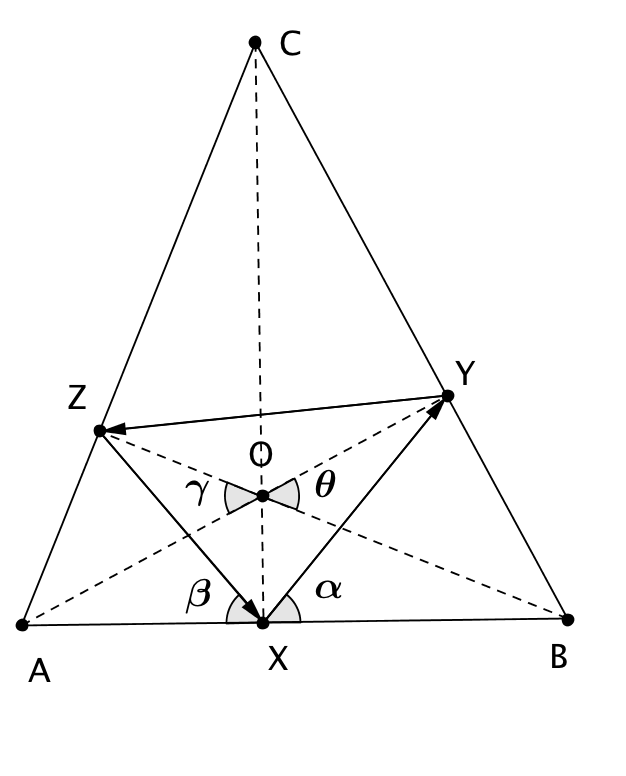}
  \end{center}
\vspace{-20pt}
  \caption{Fagnano Path}
\label{fig:Fagnano}
\vspace{10pt}
\end{wrapfigure}

Next, let us look at the famous Fagnano orbit. For any acute triangle, we can connect the three feet of the altitudes, as in Figure ~\ref{fig:Fagnano}. Then the orbit $XY, YZ, ZX$ is actually a periodic billiard path. The proof is as follows:

Clearly points $X, Y, B, O$ lies on a circle, and $A,X,O,Z$ on another circle. So $\alpha = \theta=\gamma=\beta$. Hence, the path obey the law of reflection at the point $X$, and similarly at $Y, Z$. 

\begin{wrapfigure}{r}{5.5cm}
\vspace{-30pt}
  \begin{center}
    \includegraphics[width=5.5cm]{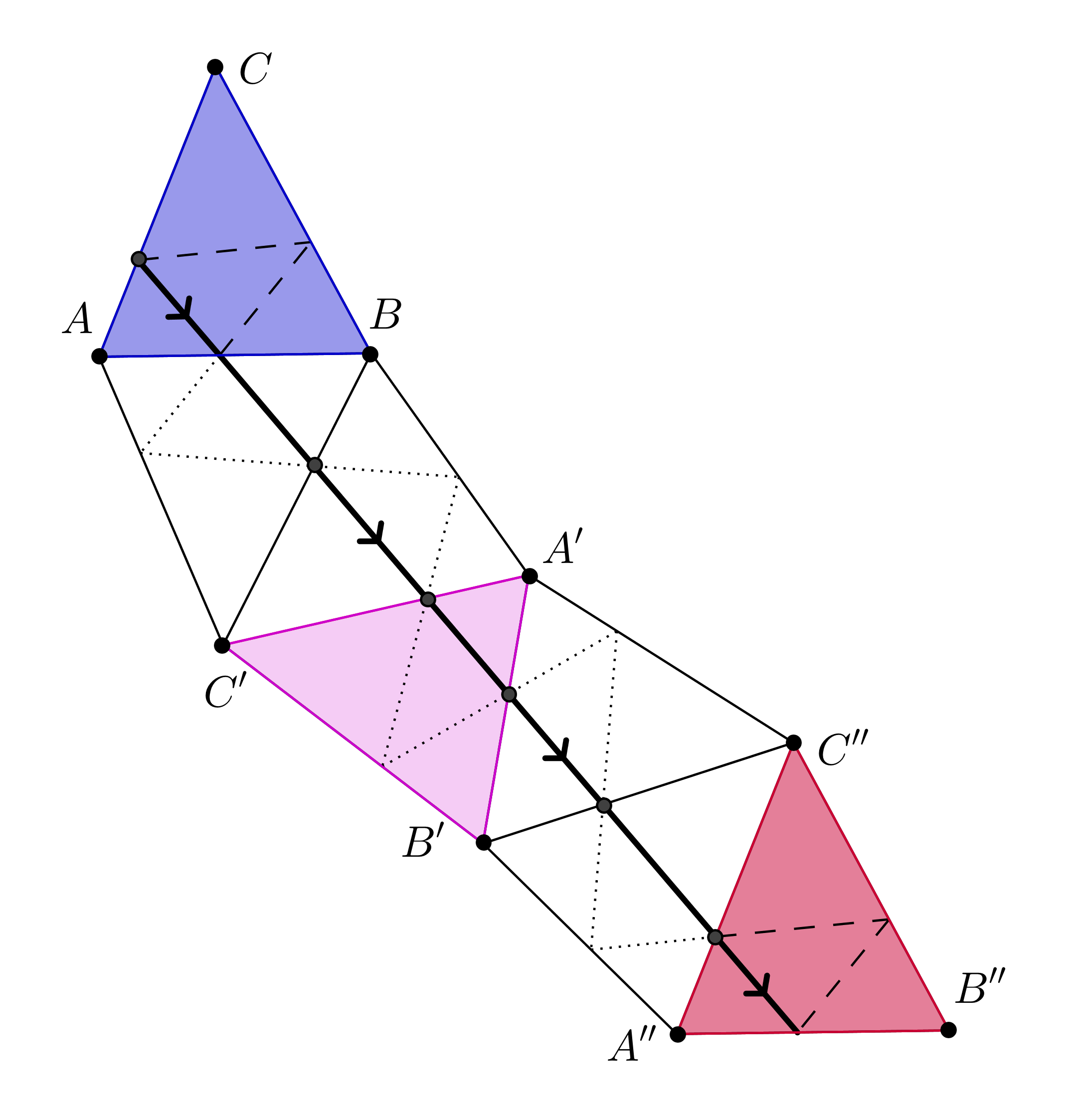}
  \end{center}
\vspace{-20pt}
  \caption{Unfolding for Fagnano path}
\label{fig:Fagnano Unfolding}
\vspace{-30pt}
\end{wrapfigure}

By labelling $AB,BC,CA$ using $0,1,2$ respectively, we can ``represent'' this periodic path by the word $012$ (which has length 3). However, after 3 reflections in the unfolding(see Figure~\ref{fig:Fagnano Unfolding}), the resultant polygon ($A'B'C'$) is not a translation of the original polygon ($ABC$). 

The reason is that each reflection of the polygon changes its orientation. In order to have a translation, we need an even number of reflections. So if a periodic path can be ``represented'' by a word $P$ of odd length (in the example above, $012$), then we define its orbit type to be $P^2$($012012$), so that the first polygon ($ABC$) and last polygon ($A''B''C''$) in the unfolding are related by a translation.

Of course, we can always reflect a polygon according to some arbitrary word. Then one important question is, given a word $W$, can we always find a corresponding periodic billiard path in some given polygon $P$? We have the following Lemma:

\begin{figure}[h]
\centering
\vspace{-15pt}
\subfloat[]{\includegraphics[height=4cm]{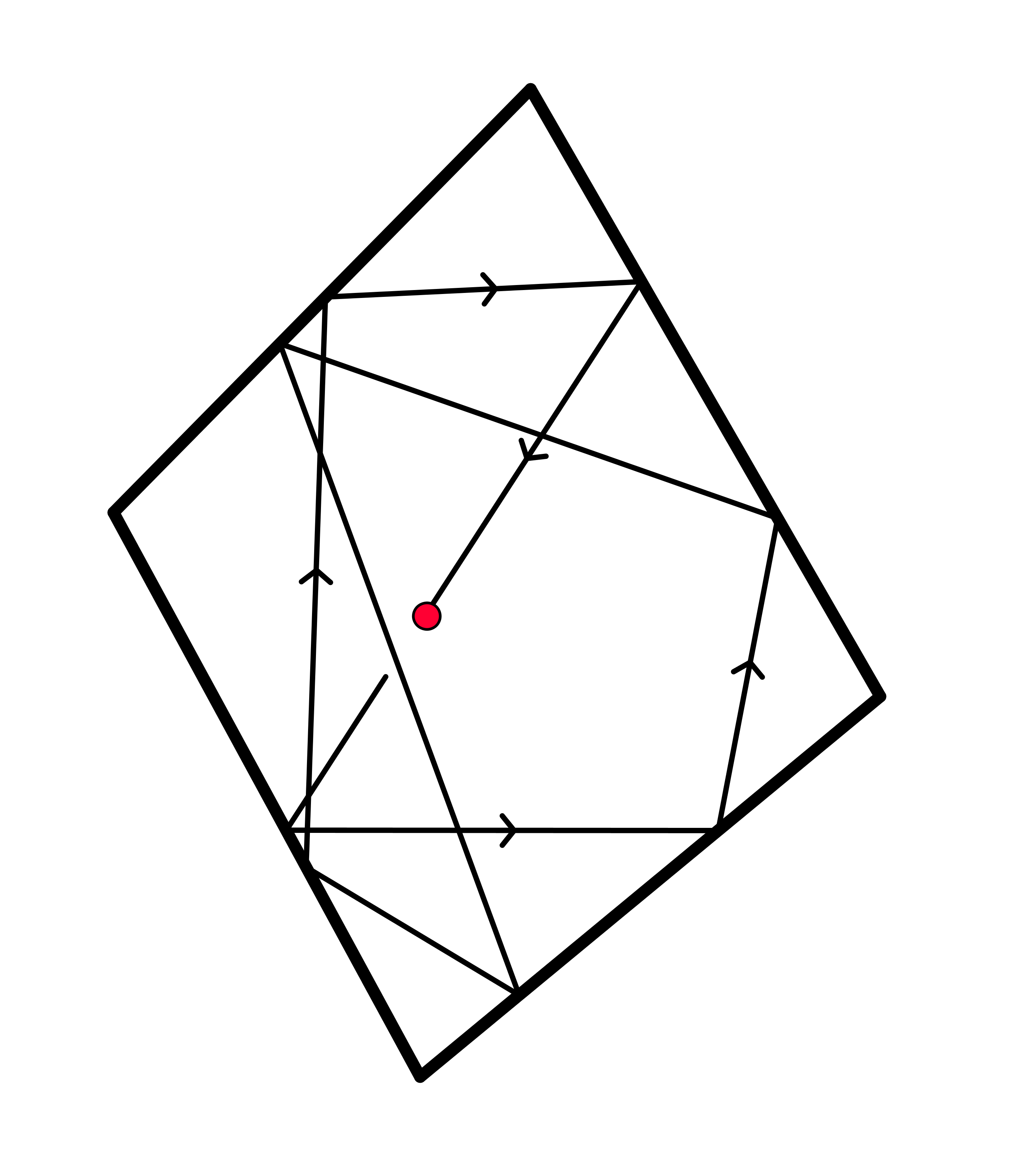}\label{fig:topvertices2}}
\subfloat[]{\includegraphics[height=5cm]{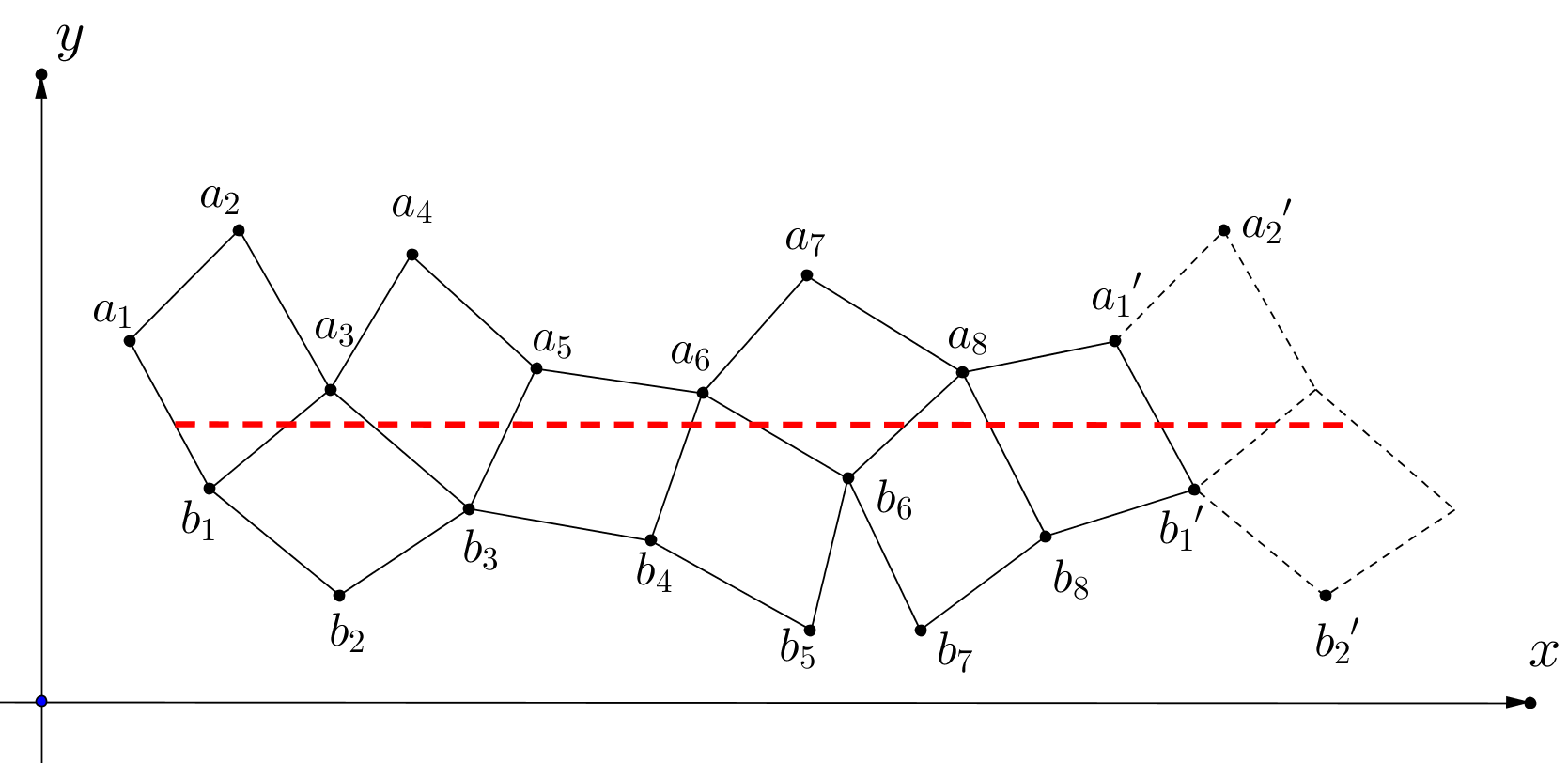}\label{fig:topvertices1}}
\caption{}
\label{fig:topvertices3}
\end{figure}
\begin{Lemma}
For an even word $W$ and a polygon $P$, there exists a periodic billiard path on $P$ with the orbit type $W$ if and only if:
\begin{enumerate}
\item The first and the last polygon in the unfolding $U(W,P)$ are related by a translation. AND
\item There exists a straight line segment, parallel to the direction of the translation, that stays within the unfolding and does not touch any of the vertex.
\end{enumerate}
\label{Lemma:UnfoldingLemma}
\end{Lemma}

To illustrate the second condition, let's look at Figure ~\ref{fig:topvertices1}, which is the unfolding for the periodic billiard path in Figure ~\ref{fig:topvertices2}. Assuming Condition 1 is satisfied, then a polygon has this periodic billiard path if and only if we can ``fit'' the red dashed line into the corresponding unfolding. (i.e. in Figure ~\ref{fig:topvertices1}, if we choose the direction of the translation as our $x$-axis, then we need the $y$-coordinate of $a_i$ to be larger than that of $b_j$ for any $i$ and $j$.)


\subsection {The space of quadrilaterals}
\label{Section:Quadrilateral Space}

\begin{wrapfigure}{r}{4.5cm}
\vspace{-20pt}
  \begin{center}
    \includegraphics[width=4cm]{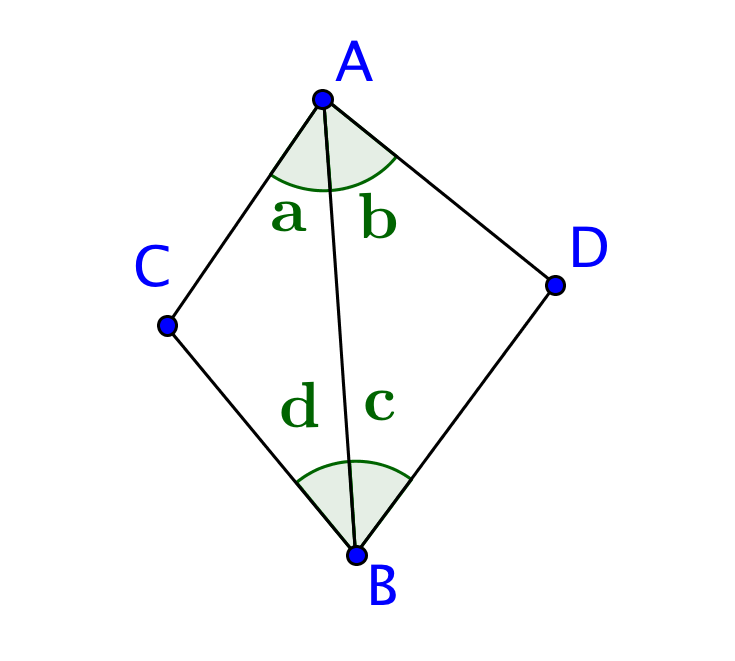}
  \end{center}
  \vspace{-20pt}
  \caption{Four Parameters}
\label{4 Dimension Quad}
\vspace{-20pt}
\end{wrapfigure}

Consider the space of all quadrilaterals (up to similarity), call it $\mathcal{Q}$. For any quadrilateral, we can split it into two triangles by drawing a diagonal. Each triangle is uniquely determined by two of its angles. So, any quadrilateral is determined by four parameters, $a,b,c,d$, as shown in Figure~\ref{4 Dimension Quad}. The parameters of the square are $(\frac{\pi}{4}, \frac{\pi}{4},\frac{\pi}{4},\frac{\pi}{4})$, and rectangles are on the line 
\begin{center} $\{(a,b,c,d):a=c, b=d, a+b=\frac{\pi}{2}\}$. \end{center}
For each orbit type $W$, we call the set of all quadrilaterals with this orbit type the \textit{orbit tile} $O(W)$ of the orbit type. 

\begin{Def}
A orbit type $W$ is called \textbf{stable} if the orbit tile $O(W)$ is a nonempty open set. Otherwise, $W$ is unstable.
\label{Def:2.2.1}
\end{Def}

We will use the following criterion for the stability of an orbit type. See \cite{Tabachnikov}

\begin{Lemma}
For $O(W) \neq\emptyset$, an orbit type W is {\normalfont stable} if and only if the number of times each sides of the polygon appears in an odd position in the orbit type equals the number of times it appears in an even position.
\label{Lem:Stable Odd Even}
\end{Lemma}

For example, the orbit type $012021$ is not stable, because $1$ appears twice in even positions (and does not appear at odd position). Also, $012012$ is stable.

Now we shall restate our main theorems more precisely. 

\begin{Def}
For $\varepsilon>0$, a quadrilateral with parameter $(a_1,a_2,a_3,a_4)$ is called $\varepsilon$-near  square if $\left|a_i-\frac{\pi}{4}\right|<\varepsilon$ for any $i$.
\label{OnlyDefinition}
\end{Def}

\begin{Theo}
The set of all $\frac{\pi}{107}$-near square quadrilaterals is covered by finitely many orbit tiles.
\label{Thrm:epsilon107}
\end{Theo}

\begin{Theo}
An open neighborhood of the rectangle line in the space $\mathcal{Q}$ can be covered by finitely many orbit tiles.
\label{Thrm:Rectangle Cover}
\end{Theo}

\section{ Proof of Theorem ~\ref{Thrm:epsilon107}}
\subsection{$\alpha$-$\beta$ plane}

\begin{figure}[h!]
\begin{center}
\subfloat[]{\includegraphics[width=3.5cm]{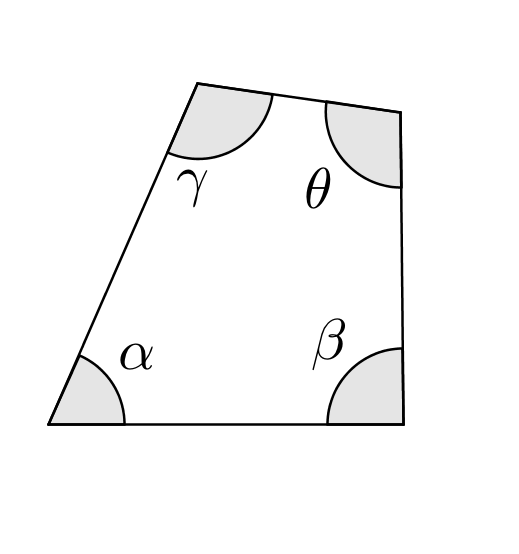}\label{fig:4angle}}
\subfloat[]{\includegraphics[width=4.5cm]{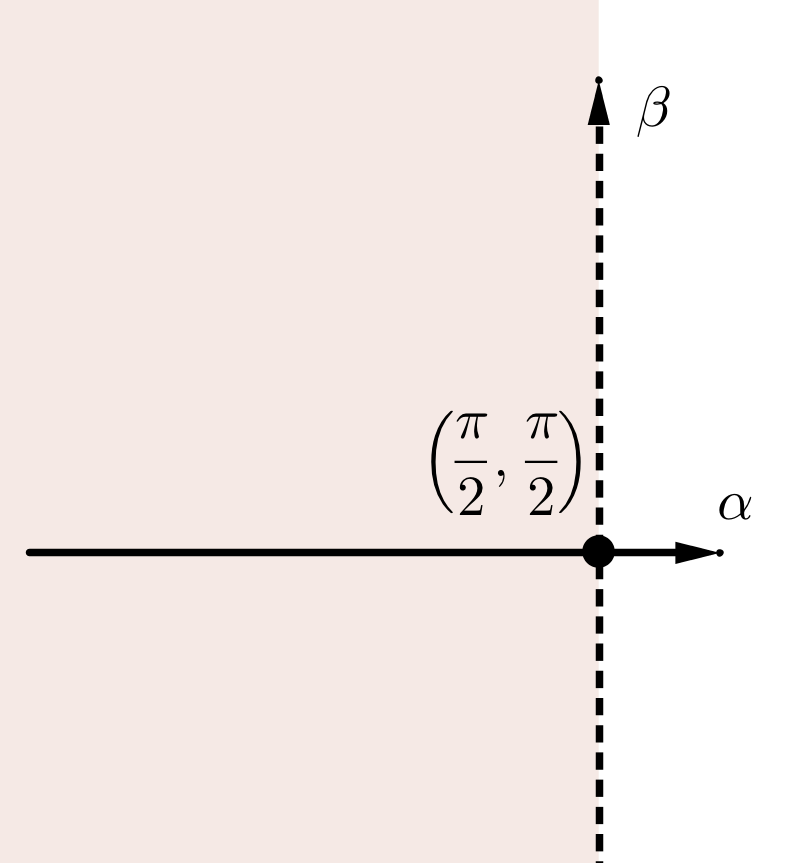}\label{fig:lefthalfplane}}
\subfloat[]{\includegraphics[width=4.5cm]{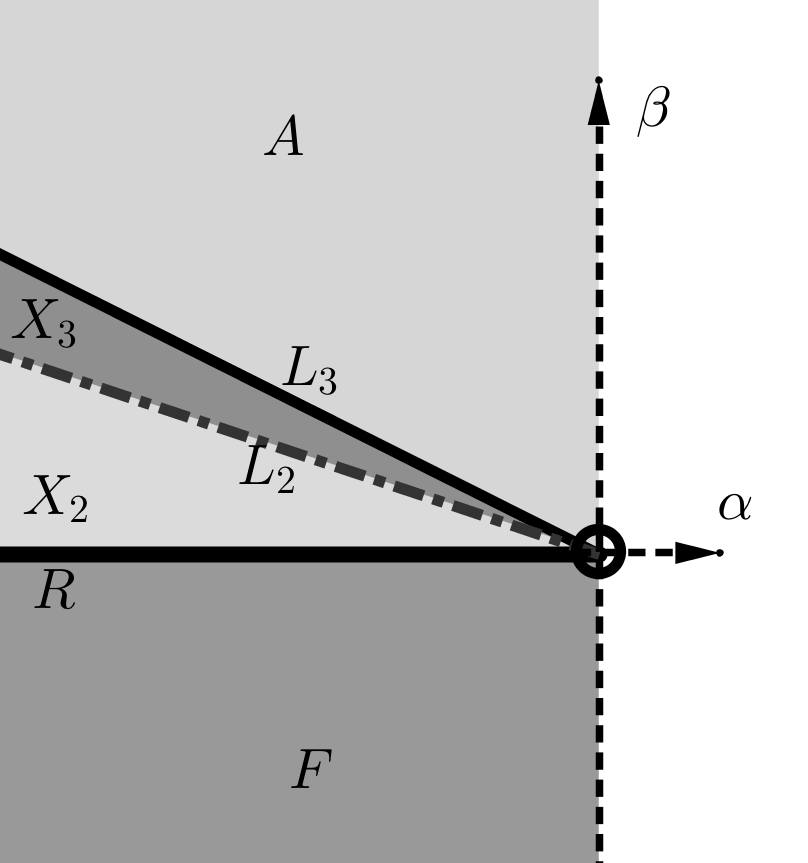}\label{fig:choplefthalf}}
\end{center}
\vspace{-10pt}
\caption{}
\vspace{-10pt}
\end{figure}

In this section, we will introduce a way of thinking about quadrilaterals. This will help us prove Theorem ~\ref{Thrm:epsilon107}.

Let $q$ be a quadrilateral. If $q$ is not rectangle, then we can choose an acute angle $\alpha$. For the two angles adjacent to $\alpha$, we pick the smaller one (or either one, if they are equal), and call it $\beta$, as shown in Figure ~\ref{fig:4angle}. If $q$ is a rectangle, then we let $\alpha=\frac{\pi}{2}, \beta=\frac{\pi}{2}$. Then, every quadrilateral will be characterized by some point on the $\alpha$-$\beta$ plane, which consists of the left half plane $\left(\alpha < \frac{\pi}{2}\right)$ and the origin $\left(\alpha=\frac{\pi}{2},\beta=\frac{\pi}{2}\right)$, as shown in Figure ~\ref{fig:lefthalfplane}.

Note that a quadrilateral $q$ being close to the $\left(\frac{\pi}{2},\frac{\pi}{2}\right)$ point on the $\alpha$-$\beta$ plane does not tell anything about the ``near-squareness'' of the quadrilateral. For example, the point $\left(\frac{\pi}{2},\frac{\pi}{2}\right)$ represents the set of all rectangles (including those rectangles that are not ``near-square'' at all). 

To prove Theorem ~\ref{Thrm:epsilon107}, we chop the left half plane into 8 regions. We show that an $\frac{\pi}{107}$-near square $q$ has a periodic billiard path whose orbit type only depends on the region.

\begin{center}
\begin{tabular}{|l|c|c|c|}
\hline
Region in $\alpha$-$\beta$ plane & in Figure ~\ref{fig:choplefthalf} &orbit type & note 
\\ \hline $\left(\frac{\pi}{2},\frac{\pi}{2}\right)$ & origin & $0202$ & Rectangles
\\ \hline $\alpha<\frac{\pi}{2},\beta<\frac{\pi}{2}$ & $F$ & $012012$ & Section ~\ref{sec:acute}
\\ \hline $\alpha<\frac{\pi}{2},\beta=\frac{\pi}{2}$ & $R$ &$012021$ & Section ~\ref{sec:right}
\\ \hline $\alpha<\frac{\pi}{2},\alpha+2\beta>\frac{3}{2}\pi$ & $A$ &$01203213$ &  Section ~\ref{sec:A}
\\ \hline $\alpha<\frac{\pi}{2},\alpha+3\beta<2\pi, \beta>\frac{\pi}{2}$ & $X_2$ & $0(13)^2013131310(31)^203$  & Section ~\ref{sec:X}
\\ \hline $\alpha<\frac{\pi}{2},\alpha+3\beta>2\pi,\alpha+2\beta<\frac{3}{2}\pi$ & $X_3$ & $0(13)^30(13)^2131(31)^20(31)^303$ & Section ~\ref{sec:X}
\\ \hline  $\alpha<\frac{\pi}{2},\alpha+3\beta=2\pi$ & $L_2$ & $01313103$ & Section ~\ref{sec:Y}
\\ \hline  $\alpha<\frac{\pi}{2},\alpha+2\beta = \frac{3}{2}\pi$ & $L_3$ & $0131313103$ & Section ~\ref{sec:Y}
\\ \hline
\end{tabular}
\end{center}

\subsection{Adjacent acute angles}
\label{sec:acute}

In this section, we will prove the following:

\begin{Prop}
If $q$ is $\frac{\varepsilon}{12}$-near square, and it has two adjacent acute angles, then $q$ has a periodic billiard path.
\label{Prop:2acute}
\end{Prop}

\begin{figure}[h]
\centering
\vspace{-15pt}
\subfloat[]{\includegraphics[width=5cm]{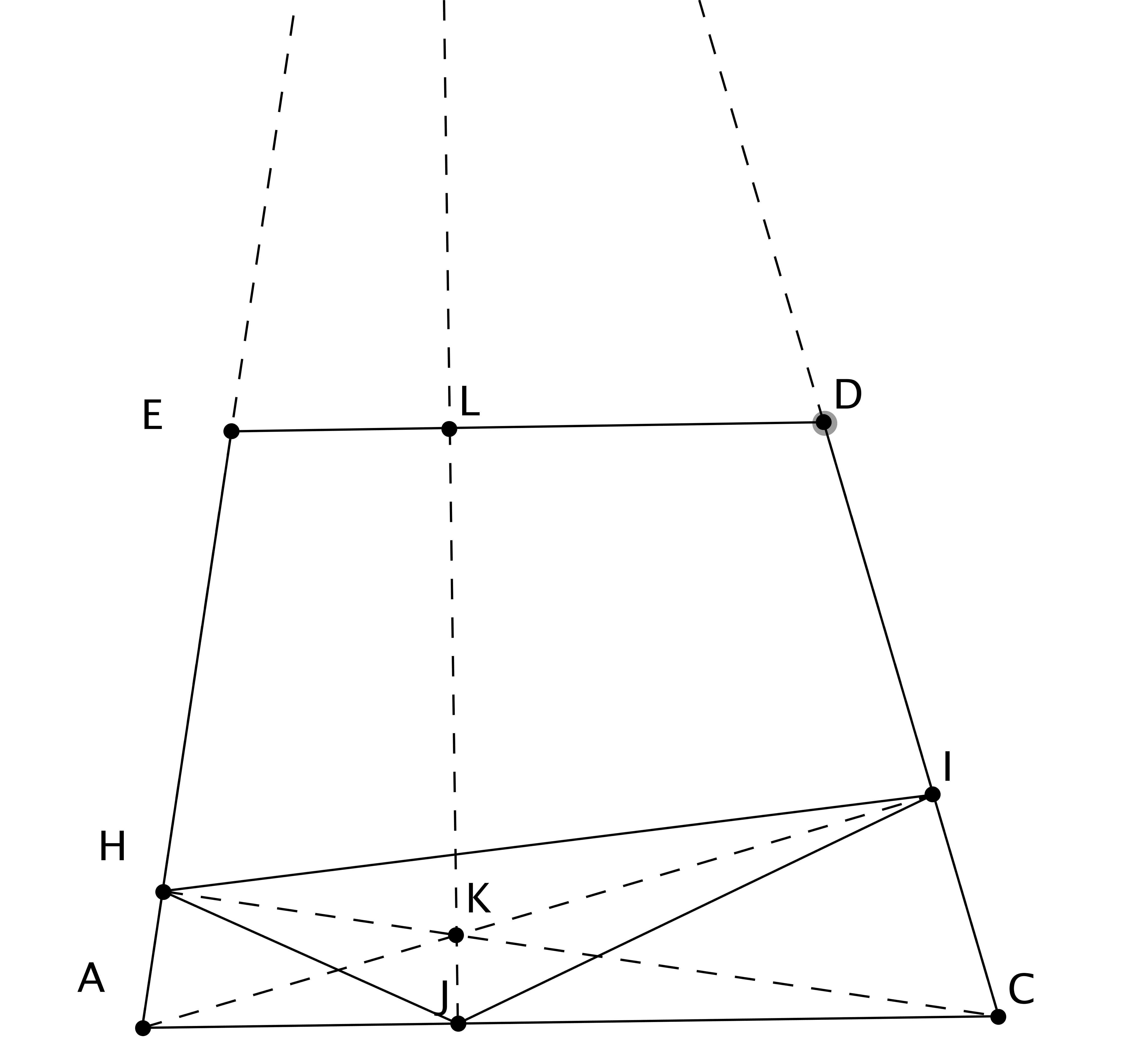}\label{fig:Fagnano Near Square}}
\subfloat[]{\includegraphics[width=6cm]{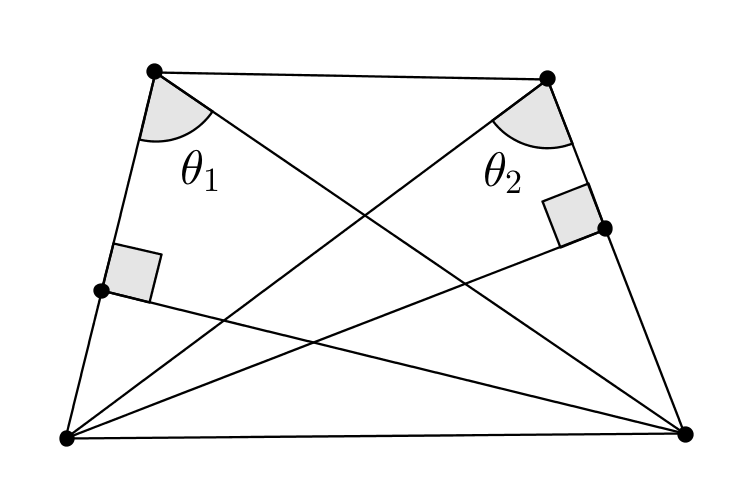}\label{fig:Fagnano bound}}
\vspace{-5pt}
\caption{}
\vspace{-5pt}
\end{figure}

We know that in every acute triangle, there's a periodic billiard path, the Fagnano orbit. The orbit type is 012012. Now suppose a near square have two adjacent acute angles. Then we can complete this near square to an acute triangle, and then find the Fagnano orbit, as shown in Figure~\ref{fig:Fagnano Near Square}. Then since the quadrilateral is near square, two of the altitudes have to be really flat, so the Fagnano orbit stays inside this near square. This gives a periodic billiard path of the near square, and let us give this orbit a name, orbit F.

\begin{wrapfigure}{r}{4cm}
\vspace{-30pt}
  \begin{center}
    \includegraphics[width=4cm]{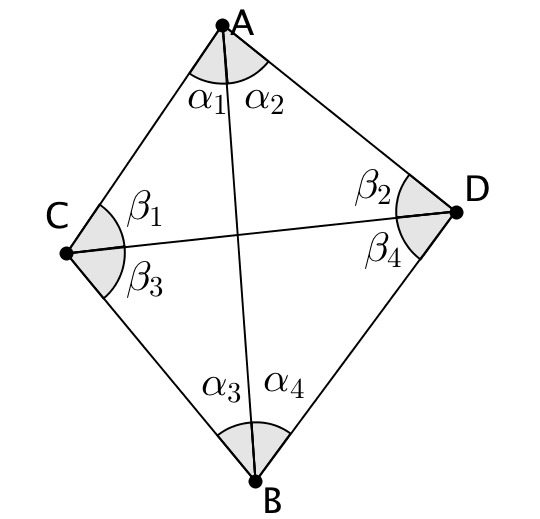}
  \end{center}
  \caption{}
\label{fig:4alpha4beta}
\vspace{-40pt}
\end{wrapfigure}

So, the central question now is how near square this path require. Let us consider Figure~\ref{fig:Fagnano bound}. The critical condition is clearly that $\theta_1$ and $\theta_2$ as shown are smaller than $\frac{\pi}{2}$. We will need the following:

\begin{Lemma}

Let $ABCD$ be a quadrilateral as shown in Figure~\ref{fig:4alpha4beta}. Suppose $|\alpha_i-\frac{\pi}{4}|<\varepsilon$ for any $i$, and $\varepsilon<\frac{\pi}{12}$, then $|\beta_j-\frac{\pi}{4}|<3\cdot\varepsilon$ for any $j$.
\label{Lem:3 Times Sigma}
\end{Lemma}

The proof of this lemma is trigonometric (using Sine Rule) and some estimates using Mathematica. We will leave the proof of this lemma to the reader. 

\begin{proof}[Proof of Prop ~\ref{Prop:2acute}]
Suppose quadrilateral $q$ is $\frac{\pi}{12}$ near square. Let $\theta_1$ and $\theta_2$ be the angle shown in figure ~\ref{fig:Fagnano bound}. By the above discussion, it suffice to show that $\theta_i<\frac{\pi}{2}, i=1,2$. Depending on the choice of diagonal used in the parameterization, one of the two angles $\theta_i$ satisfies $\left|\theta_i -\frac{\pi}{4}\right|<\frac{\pi}{2}$. By the lemma, the other angle satisfies $\left|\theta_j -\frac{\pi}{4}\right|<\frac{\pi}{4}$
\end{proof}
\subsection{Acute angle adjacent to a right angle}
\label{sec:right}
\begin{wrapfigure}{r}{4cm}
\vspace{-50pt}
  \begin{center}
    \includegraphics[width=4cm]{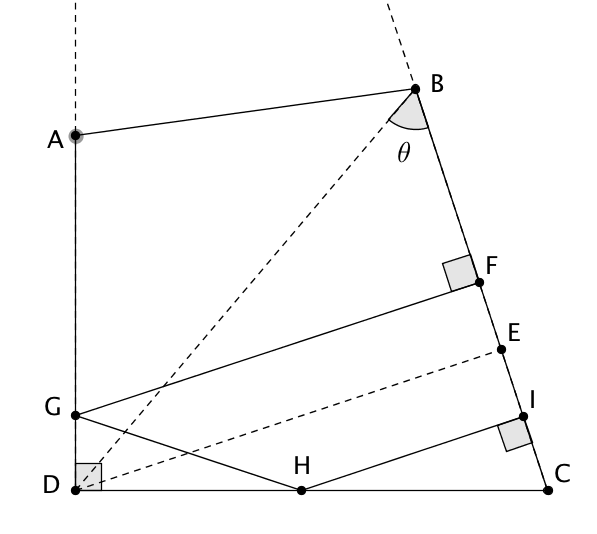}
  \end{center}
  \caption{}
\label{fig:Acute Right}
\vspace{-90pt}
\end{wrapfigure}

%
%
\begin{Prop}
If $q$ is $\frac{\pi}{12}$-near square, and it has a right angle adjacent to an acute angle, then $q$ has a periodic billiard path.
\label{Prop:AcuteRightQuad}
\end{Prop}

The proof is the same as the proof of proposition ~\ref{Prop:2acute}. In this case, we need to show the angle $\theta$ of Figure ~\ref{fig:Acute Right} is acute.


\subsection{Opposite acute angles}
\label{sec:A}

\begin{figure}[h]
\centering
\includegraphics[width=8cm]{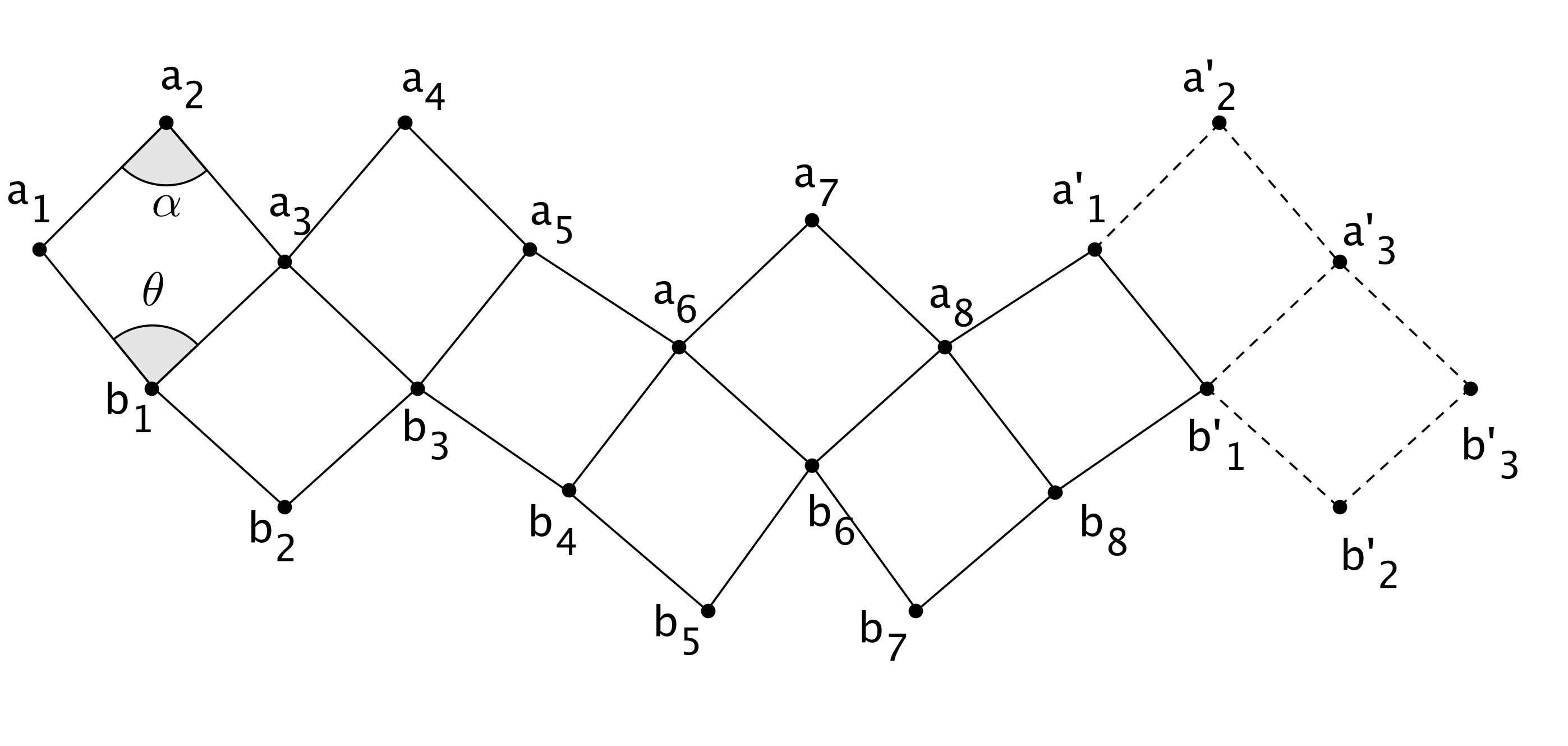}
\vspace{-5pt}
\caption{Orbit A}
\label{fig:Orbit A}
\end{figure}

Here we are going to study a specific periodic billiard path with combinatorial orbit 01203213, whose unfolding is illustrated in Figure ~\ref{fig:Orbit A}. The last two dotted quadrilateral is a translation of the first two quadrilateral. This orbit is clearly stable by Lemma ~\ref{Lem:Stable Odd Even}, and the direction of translation (and thus the direction of the billiard path) is the same as the line connecting $b_1$ and $b_1'$. If we let this be the $x$-axis, and let $y(b_i)$ denote the $y$-coordinate of $b_i$ (similarly for $a_i$), then following the discussion after Figure ~\ref{fig:topvertices3}, we would like to prove:

\begin{Prop}
If the quadrilateral is $\frac{\pi}{30}$ near square, $\alpha, \theta$ in Figure ~\ref{fig:Orbit A} are acute, the other two angles are obtuse, then 
\begin{equation*}
\max_{j} y(b_j) < \min_{i} y(a_i)
\end{equation*}
and hence the quadrilateral has this periodic path.
\label{OrbitAProp}
\end{Prop}

Let us start with some observations.

\begin{Remark}
If our quadrilateral is $\frac{\pi}{12}$ near square, then $\max_{j} y(b_j) = \max \{y(b_1),y(b_3),y(b_6)\}$, and $\min_{i} y(a_i) = \min \{y(a_3),y(a_6),y(a_8)\}$
\end{Remark}
\begin{proof}
Since our quadrilateral is convex, we know $y(a_2)$ is definitely larger than the smaller one of $y(a_1)$ and $y(a_3)$, so we don't need to consider the vertex $a_2$. Similarly we can eliminate vertices $a_4, a_7, b_2, b_5, b_7$.

Now suppose that our quadrilateral is $\dfrac{\pi}{12}$-near square, then we know $\angle b_3 a_5 a_6 < \dfrac{2\pi}{3}$, and $\angle a_3 a_5 b_3 < \dfrac{\pi}{3}$. So, we know $\angle a_3 a_5 a_6 < \pi$. So again by convexity, $y(a_5)$ is larger than the smaller one of $y(a_3)$ and $y(a_6)$. In the same fashion, we can eliminate vertices $a_1, b_4, b_8$.
\end{proof}

So, for Proposition ~\ref{OrbitAProp}, all that remains is to show: 
\begin{equation*}
\max \{y(b_1),y(b_3),y(b_6)\} < \min \{y(a_3),y(a_6),y(a_8)\}
\end{equation*}

Note that $y(a_1')=y(a_1), y(b_1')=y(b_1)$. In this section, we use the word `angular slope' to refer to the angle at which the line meets with the $x$-axis, i.e. `angular slope' $= \tan^{-1} (m)$ with values in $\left(-\frac{\pi}{2},\frac{\pi}{2}\right)$, where $m$ is the gradient of the non-vertical line. 

\begin{Prop}
If the quadrilateral is $\dfrac{\pi}{4}$-near square, then 
$min(y(a_6), y(a_8)) > max(y(b_1), y(b_3))$ if and only if both $\alpha$ and $\theta$ are acute.
\label{Prop:Opposite Acute}
\end{Prop}

\begin{proof}
Let us restrict our attention to $b_1,b_3,a_6,a_8,b_1',b_3'$, as shown in Figure \ref{fig:Pic1}, \ref{fig:Pic2}, \ref{fig:Pic3}, \ref{fig:Pic4}. One can tell from Figure ~\ref{fig:Orbit A} that for the angles \textit{drawn in the Figures below}, $\angle b_1 b_3 a_6 = \angle a_6 a_8 b_1' = 2\alpha$, and $\angle b_3 a_6 a_8 = \angle a_8 b_1' b_3' = 2\theta$. Note that $b_1 b_3, b_3 a_6, a_6 a_8, a_8 b_1', b_1' b_3'$ all have the same length, since they are the same diagonal of the same quadrilateral. Let us connect $b_1 a_6$ and $a_6 b_1'$. Clearly, $\bigtriangleup b_1 a_6 b_3\cong \bigtriangleup b_1' a_6 a_8$ are the same isoceles triangle. So, $ \angle b_1 a_6 b_3 = \angle b_1' a_6 a_8$, and hence $\bigtriangleup b_1 a_6 b_1'$ is an isosceles triangle with top angle $2\theta$ or $2\pi-2\theta$. So if $\theta$ is obtuse or right angled, then $y(a_6)\leq0=y(b_1)$. Similarly, if $\alpha$ is obtuse or right angled, then observe from $\bigtriangleup b_3 a_8 b_3'$ that $y(b_3)\geq y(a_8)$. 

If both $\alpha$ and $\theta$ are acute (see Figure ~\ref{fig:Pic4}), then $y(a_6)>y(b_1)$ and $y(b_3)< y(a_8)$. Hence, all that remains is to show $y(a_6)>y(b_3)$ and $y(a_8)>y(b_1)$. Using the isosceles triangles $\bigtriangleup b_1 a_6 b_1'$ and $\bigtriangleup b_1 b_3 a_6$, we see that the slope of the line $b_1 b_3$ is $\dfrac{\pi-2\theta}{2}-\dfrac{\pi-2\alpha}{2} = \alpha - \theta$. (Note: if $\alpha < \theta$, then $y(b_3)<y(b_1)$. This following argument is still valid after redrawing Figure ~\ref{fig:Pic4} accordingly) 

One can also see that the slopes of $b_3 a_6, a_6 a_8, a_8 b_1', b_1' b_3'$ are $\pi-\alpha -\theta, \theta -\alpha, \alpha +\theta-\pi, \alpha - \theta$ respectively. 

So, to show $y(a_6)>y(b_1)$ and $y(a_8)>y(b_1')$ is equivalent to showing the slope of $b_3 a_6$ is bigger than zero and the slope of $a_8 b_1'$ is less than zero (i.e. $0< \pi-\alpha -\theta<\frac{\pi}{2} $ and $0> \alpha +\theta-\pi>-\frac{\pi}{2} $). Note that $\alpha, \theta < \frac{\pi}{2}$ implies $\pi-\alpha -\theta >0$. Furthermore, if the quadrilateral is $\frac{\pi}{4}$-near square, then $\alpha, \theta > \frac{\pi}{4}$, which implies $ \pi-\alpha -\theta<\frac{\pi}{2}$. This completes the proof.
\end{proof}

\begin{Prop}
When the quadrilateral is $\dfrac{\pi}{12}$-near square, then $y(a_3) > \max(y(b_1), y(b_3))$, and $\min(y(a_6), y(a_8)) > y(b_6)$.
\label{Prop:3.4.2}
\end{Prop}

\begin{proof}
We know from Proposition ~\ref{Prop:Opposite Acute} that the slope of $b_1 b_3$ is $\alpha-\theta$, which is between $\dfrac{\pi}{6}$ and $-\dfrac{\pi}{6}$ when the quadrilateral is $\dfrac{\pi}{12}$-near square. As $\angle a_3 b_1 b_3$ is between $\dfrac{\pi}{6}$ and $\dfrac{\pi}{3}$, we see that the slope of $b_1 a_3$ is between $\dfrac{\pi}{2}$ and $0$. So it has positive slope, and $y(a_3)>y(b_1)$. By basically the same argument, we can show that $y(a_3)>y(b_3), y(a_6)>y(b_6), y(a_8)>y(b_6)$. 
\end{proof}
\pagebreak
\begin{figure}[h]
\centering
\vspace{-35pt}
\subfloat[]{\includegraphics[width=7cm]{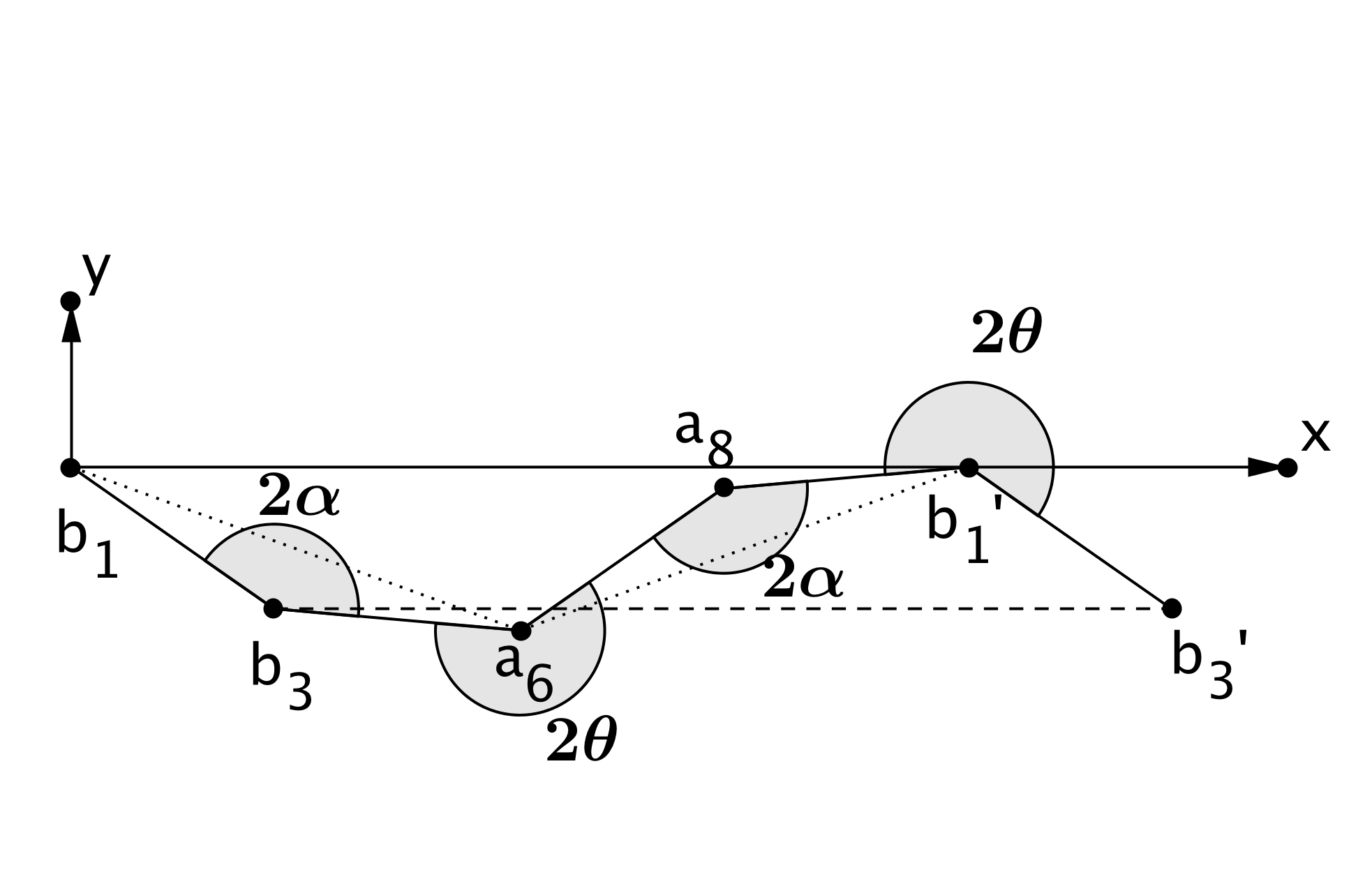}\label{fig:Pic1}}
\subfloat[]{\includegraphics[width=7cm]{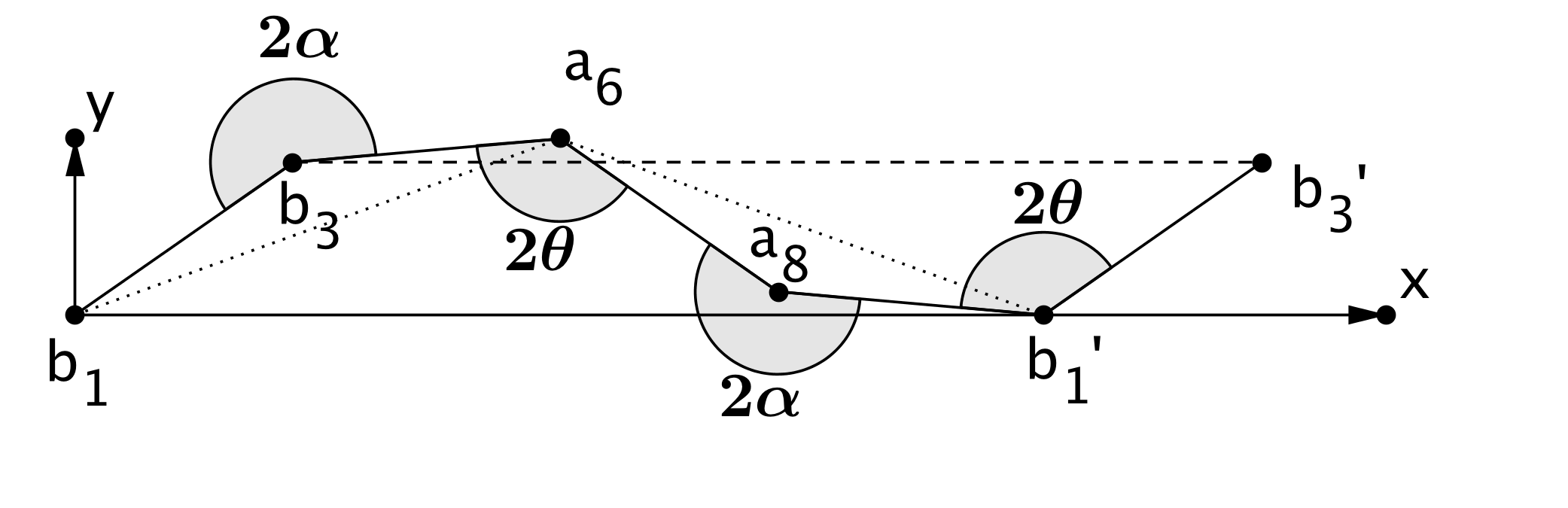}\label{fig:Pic2}}\\
\subfloat[]{\includegraphics[width=7cm]{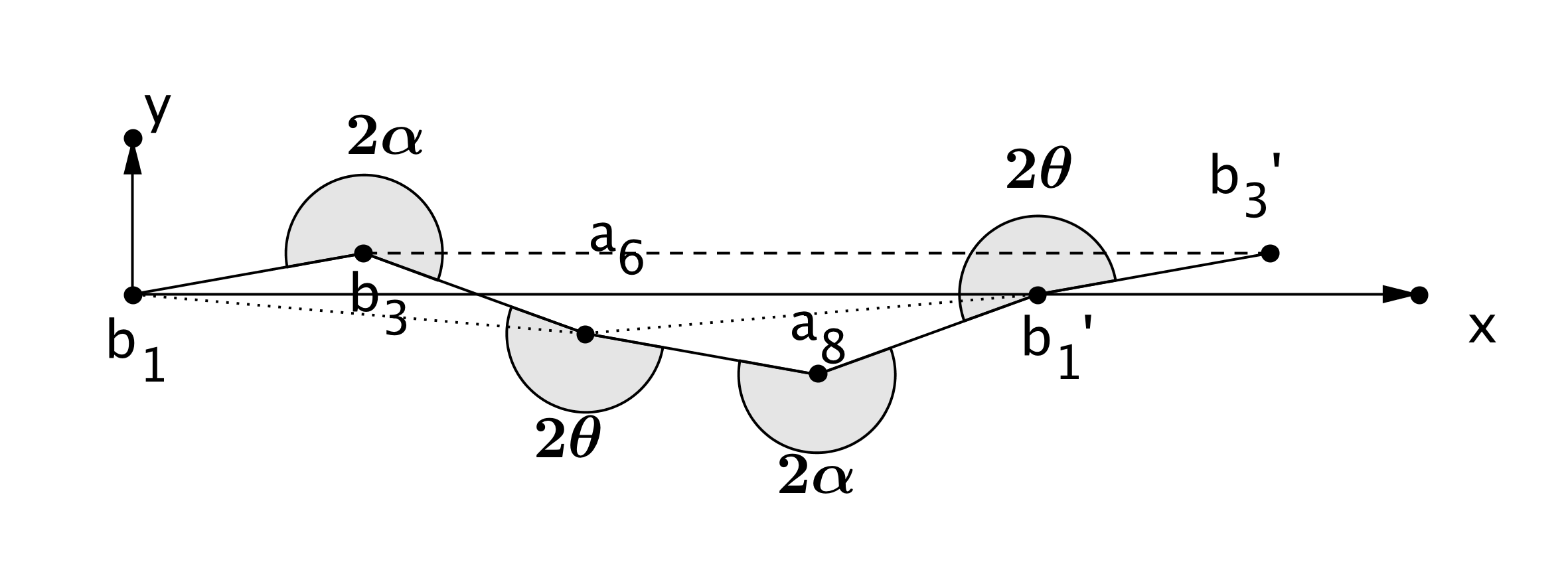}\label{fig:Pic3}}
\subfloat[]{\includegraphics[width=7cm]{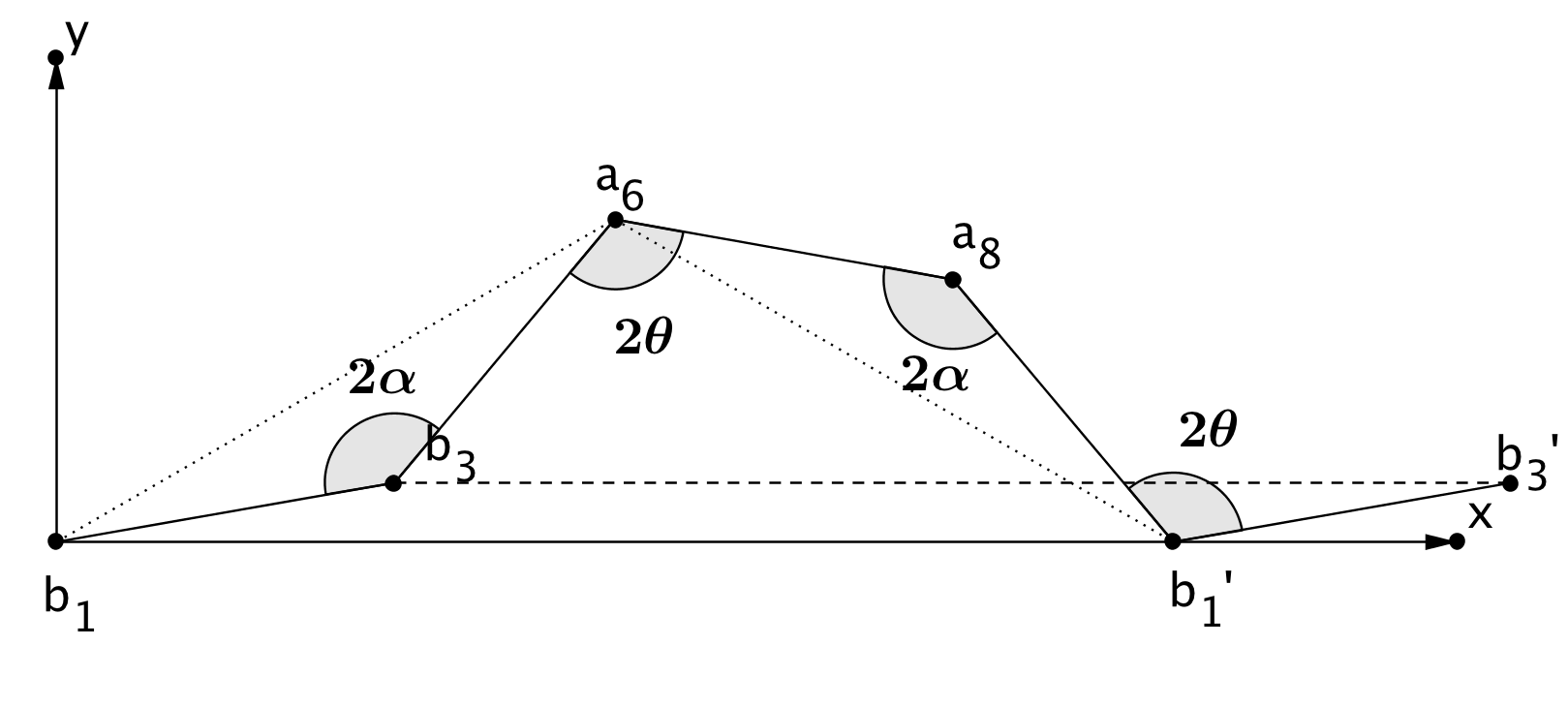}\label{fig:Pic4}}
\caption{(a)$\alpha<\frac{\pi}{2}$, $\theta>\frac{\pi}{2}$, (b) $\alpha>\frac{\pi}{2}$, $\theta<\frac{\pi}{2}$ (c) $\alpha>\frac{\pi}{2}$, $\theta>\frac{\pi}{2}$ (d) $\alpha<\frac{\pi}{2}$, $\theta<\frac{\pi}{2}$}
\end{figure}
\begin{Lemma}
In an $\epsilon$-near square, if one edge has length 1, then the length of the adjacent edge is between $\tan\left(\dfrac{\pi}{4}-\epsilon\right) \cos \left(2 \epsilon\right)$ and $\dfrac{\tan\left(\frac{\pi}{4}+\epsilon\right)}{\cos (2 \epsilon)}$.
\label{Lem:Adjacent Edge Bound}
\end{Lemma}
\begin{proof}
WLOG, for Figure ~\ref{fig:4alpha4beta}, let $|AC|=1$. Then by sine rule:
\begin{equation*}
\dfrac{BC}{AC}=\dfrac{BC}{1}=\dfrac{\sin \alpha_1}{\sin \alpha_2},
\hspace{0.5cm}
\dfrac{AD}{AC}=\dfrac{AD}{1}=\dfrac{\sin \beta_1}{\sin \beta_2}=\dfrac{\sin\alpha_4}{\sin(\alpha_2+\alpha_4)}\dfrac{\sin(\alpha_1+\alpha_3)}{\sin\alpha_3}
\end{equation*}
We will leave the rest of the estimation to the reader.
\end{proof}

\begin{Prop}
When the quadrilateral is $\dfrac{\pi}{30}$-near square, $\alpha$ and $\theta$ are acute, and the other two angles are obtuse, then $y(a_3)>y(b_6)$.
\label{Prop:3.4.4}
\end{Prop}

\begin{figure}[h]
\centering
\includegraphics[width=11cm]{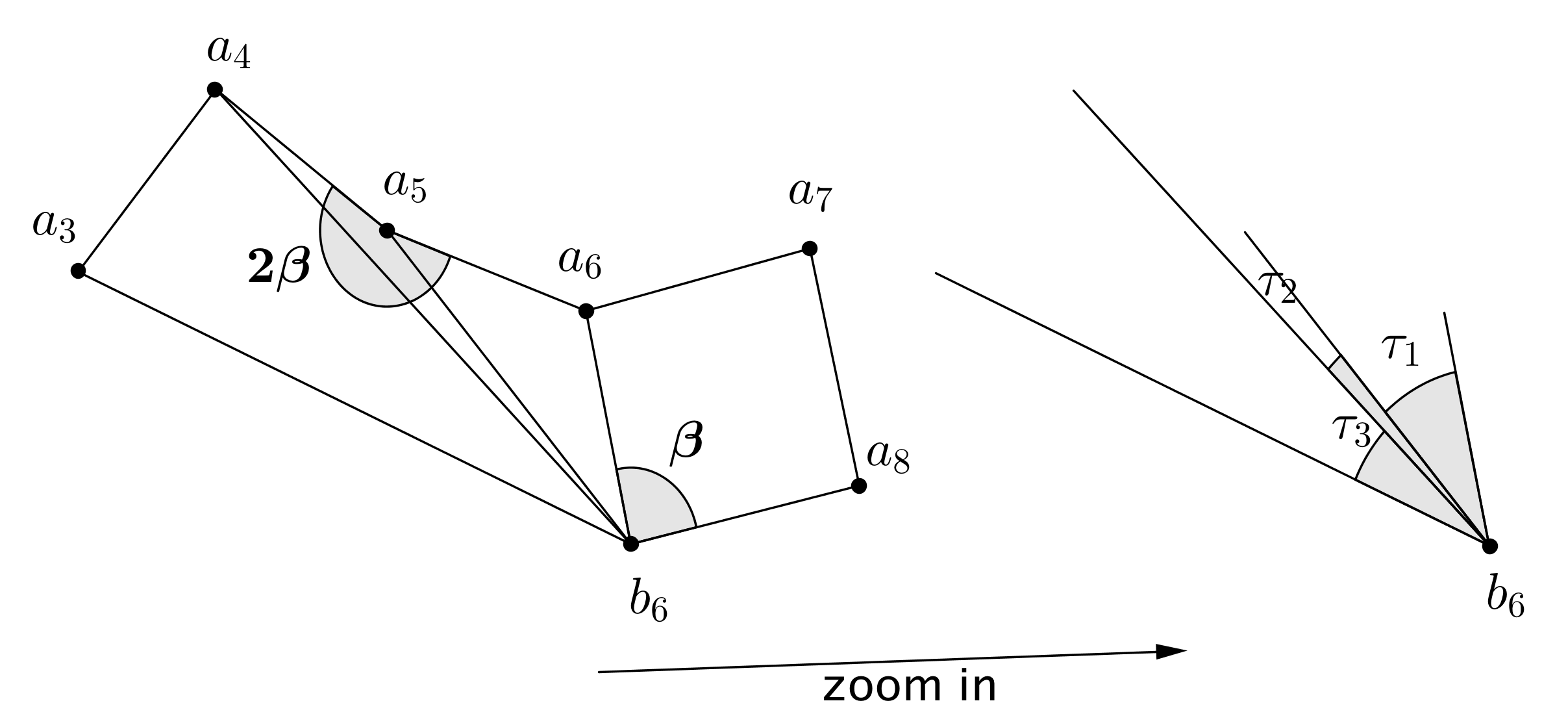}
\caption{ }
\label{fig:Extract}
\end{figure}

\begin{proof}

Clearly the line $a_3 a_3'$ is parallel to the $x$-axis. Then it clearly follows that $y(a_3)>y(b_6)$ if and only if $\angle a_3 b_6 a_3' < \pi$ (See figure ~\ref{fig:Orbit A})

Let us first look at $\angle a_6 b_6 a_3$ by breaking it down into three angles, $\tau_1, \tau_2, \tau_3,$ as shown in Figure ~\ref{fig:Extract}. Suppose our quadrilateral is $\frac{\pi}{30}$-near square, then (See figure ~\ref{fig:Orbit A}) $\theta > \frac{\pi}{2}-2\frac{\pi}{30}=\frac{13\pi}{30}$, so $\angle a_5 a_6 b_6 =2\theta> \frac{13\pi}{15}$. As $\bigtriangleup a_5 a_6 b_6$ is an isosceles triangle, we see that $\tau_1<\frac{\pi}{15}$. (also, since $\theta<\frac{\pi}{2}$, so $\tau_1 >0$)

Let $\angle a_6 b_6 a_8 = \beta$. Then $\angle a_4 a_5 a_6 = 2\beta$ as shown, and $\angle a_4 a_5 b_6 = 2\beta - \tau_1 > 2\beta - \frac{\pi}{15}$. If $\angle a_4 a_5 b_6 \geq \pi$, then $\tau_2 \leq 0$ (i.e. in the figure above, $\tau_1$ and $\tau_3$ overlap). So $\tau_1+\tau_2 \leq \tau_1 <\frac{\pi}{15}$. 

Note that $a_5 b_6 = 2 \cos(\angle a_6b_6a_5)\cdot a_6b_6 >2 \cos \frac{\pi}{15}*a_6 b_6 > a_6 b_6 = a_4 a_5$. So, if $\angle a_4 a_5 b_6<\pi$, then $a_5 b_6 > a_4 a_5$ implies  $\tau_2 < \angle a_5a_4b_6,$ hence $\tau_2<\dfrac{\angle a_5a_4b_6+\tau_2}{2}=\dfrac{\pi-\angle a_4 a_5 b_6}{2} < \dfrac{\pi-(2\beta - \frac{\pi}{15})}{2}=\frac{8\pi}{15} -\beta$. 

WLOG, let the length $a_4 a_5 = a_5 a_6 = a_6 b_6 = 1$, then we know that $a_5 b_6 > 2 \cos \frac{\pi}{15}$. Note that $\beta \in (\frac{\pi}{2},\frac{17}{30} \pi), \tau_1 \in (0,\frac{\pi}{15})$, so $\angle a_4 a_5 b_6=2\beta - \tau_1 \in (\frac{14}{15}\pi,\frac{17}{15}\pi)$ (and hence $\cos (\angle a_4 a_5 b_6) <0$). 

Then we can see that the length $a_4 b_6 > \sqrt{1^2 + (2 \cos \frac{\pi}{15})^2 -2 \cdot 1 \cdot 2 \cos \frac{\pi}{15} \cdot \cos(\frac{17}{15}\pi)}\approx 2.89852$. We also know by Lemma~\ref{Lem:Adjacent Edge Bound} that the length $a_3 a_4 < \dfrac{\tan(\frac{\pi}{4}+\frac{\pi}{30})}{\cos \frac{\pi}{15}}\approx 1.26249$. Note that $\sin\tau_3 \leq \dfrac{a_3 a_4}{a_4 b_6} <\frac{1.26249}{2.89852} \approx 0.43556$. So $\tau_3 < p=0.45066$.

So now if $2\beta - \frac{\pi}{15} \geq \pi$, then $\angle a_4 a_5 b_6 \geq \pi$, and so  $\tau_1+\tau_2+\tau_3 <\frac{\pi}{15}+p$. If $2\beta - \frac{\pi}{15}<\pi$, then $\angle a_6 b_6 a_3 =\tau_1+\tau_2+\tau_3 < (\frac{\pi}{15})+(\frac{8\pi}{15} -\beta )+p=\frac{3\pi}{5}-\beta +p$. But the latter one is clearly looser in this case, so we use the latter inequality. Similarly, one can also show that $\angle a'_3 b_6a_8< \frac{3\pi}{5}-\beta +p$. So $\angle a_3 b_6 a_3' < 2(\frac{3\pi}{5}-\beta+p)+\beta <2p+\frac{6\pi}{5}-\beta<2p+\frac{6\pi}{5}-\frac{\pi}{2} \approx 3.10043 < \pi$. Here the second to last inequality follows because $\beta$ is supposed to be obtuse. This proves the proposition.

\end{proof}

For a summary, a quadrilateral which is $\frac{\pi}{30}$-near square has orbit A if it has two opposite acute angles and two opposite obtuse angles (i.e. it is represented in the region $A$ in Figure ~\ref{fig:choplefthalf}).


\subsection{The $X$ family orbits}
\label{sec:X}

The $X$ family orbits is an infinite family of orbits, starting from $X_2$. Let us first study $X_2$, and then study the general pattern of the $X$ family orbits.

\begin{figure}[h]
\includegraphics[width=\textwidth]{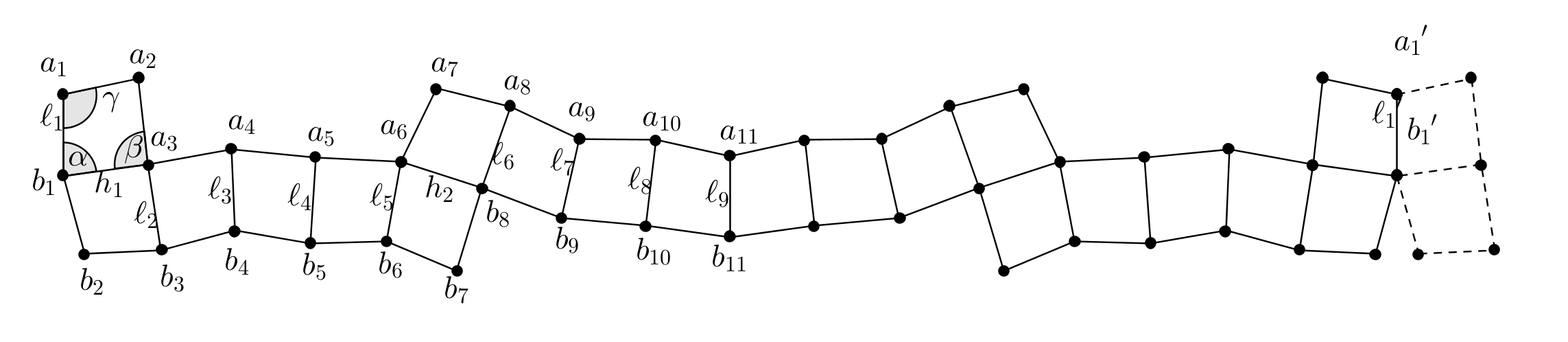}
\caption{Unfolding of $X_2$ orbit}
\label{fig:X2}
\end{figure}

Up to a relabeling of the edges, $X_2$ has the orbit type 01313013131310313103. This is clearly stable by Lemma ~\ref{Lem:Stable Odd Even}. An unfolding would look like Figure ~\ref{fig:X2}. The first obvious feature of $X_2$ is that the unfolding is symmetric about the line $\ell_9$. Note that the motion from $\ell_1$ to $\ell_1'$ is a translation. So, if we build a coordinate frame using $b_1 b_1'$ as the $x$-axis and let $\ell_1$ as the $y$-axis, then $\ell_1,\ell_9,$ and $\ell_1'$ are vertical.

Now to find the condition for this orbit to exist, we must find the condition when all top vertices are above all bottom vertices. i.e. in Figure ~\ref{fig:X2}, we want to show:
\begin{equation*}
\displaystyle \max_j y(b_j) <\min_i y(a_i)
\end{equation*}
Since everything is symmetric about $\ell_9$, we only need to consider the left half part of the unfolding. In the following proof, we shall use $v | \ell | w$ to denote the fact that the objects $v$ and $w$ are symmetric about the line $\ell$. Furthermore, we shall use $s(\ell)$ to denote the `angular slope' of the line $\ell$ (`angular slope' $= \tan^{-1} (m)$ with values in $\left(-\frac{\pi}{2},\frac{\pi}{2}\right)$, where $m=$ gradient of $\ell$). In addition, we use the notation $\min y(a_1,a_2,a_3)$ to denote the minimum of $\{y(a_1),y(a_2),y(a_3)\}$, and similar notation for the maximum.

\begin{Prop}
Suppose the quadrilateral is $\dfrac{\pi}{16}$-near square, then in Figure ~\ref{fig:X2}, TFAE:
\begin{enumerate}
\item $\min y(a_3, a_6)>\max(b_1, b_8)$ 
\item $s(\ell_3)<0, s(\ell_4)>0$ and $s(h_1)>0$. 
\item $\alpha <\frac{\pi}{2}, \beta >\frac{\pi}{2}, 3\beta+\alpha < 2\pi$
\end{enumerate}
\label{Prop:3.5.1}
\end{Prop}
\begin{proof}
$(1) \Rightarrow (2)$: Since $y(a_6)>y(b_1)$ and $b_1 | \ell_3 | a_6$, we see that $s(\ell_3)<0$. Since $y(a_3)>y(b_8)$ and $a_3 | \ell_4 | b_8$, we see that $s(\ell_4)>0$. Since $y(a_3)>y(b_1)$, we see that $s(h_1)>0$.

$(2) \Rightarrow (3)$: One can verify that $s(h_1)=\frac{\pi}{2}-\alpha, s(\ell_3)\equiv s(h_1)+2\beta+\alpha = \frac{\pi}{2}+2\beta (\text{mod } \pi)$, and $s(\ell_4)\equiv s(\ell_3)+\alpha+\beta = \frac{\pi}{2} +3\beta+\alpha (\text{mod } \pi)$. Assume for contradiction that $\frac{3\pi}{8} < \beta \leq \frac{\pi}{2}$ (the first inequality is due to the near-squareness). Then $\pi+\frac{1}{4}\pi < 2\beta+\frac{\pi}{2} \leq \pi +\frac{\pi}{2}$. So, after taking modulo $\pi$, we get $s(\ell_3)>0$ (which contradicts (2)). Similarly, we can argue that $\alpha<\frac{\pi}{2}$ and $3\beta +\alpha <2\pi$.

$(3) \Rightarrow (2) \& (1)$: The near square condition also gives us $\alpha>\dfrac{3\pi}{8}, \beta<\dfrac{5\pi}{8}$. So, using the conditions in (3), we can get the inequalities in $(2)$, $\ell_2, h_2$ having negative slope and $\ell_5, \ell_6, \ell_7, \ell_8$ having positive slope. Now the slope of $h_1, h_2$ makes sure that $y(a_3)>y(b_1), y(a_6)>y(b_8)$. And as $b_1 | \ell_3 | a_6, a_3 | \ell_4 | b_8$, we have $y(a_6)>y(b_1)$, and $y(a_3)>y(b_8)$. This completes our proof.
\end{proof}

\begin{Remark}
If the condition for Prop ~\ref{Prop:3.5.1} is satisfied, and the angle $\gamma$ is obtuse, then \\$\displaystyle \max_j y(b_j)=\max y(b_1,b_4,b_5,b_8),\min_i y(a_i) =\min y(a_3,a_6,a_{11})$
\end{Remark}
\begin{proof}
If $\gamma$ is obtuse, then immediately we know $y(a_2)>y(a_1), y(a_{10})>y(a_{11}),$ so vertices $a_2, a_{10}$ are eliminated. 

Now as $a_3 | \ell_3 | a_5, a_4 | \ell_4 | a_6,$ by the slope of $\ell_3$ and $\ell_4$ we must have $y(a_5)>y(a_3), y(a_4)>y(a_6)$. So we don't need to consider vertices $a_4, a_5$. We also have $b_8 | \ell_7 | b_{10}, b_9 | \ell_8 | b_{11},$ so similarly we eliminate vertices $b_{10}, b_{11}$. We also know that $\angle b_8 b_9 b_{10} = 2\alpha <2\pi$, so $y(b_9)<max(y(b_8), y(b_{10}))$, and we eliminate $b_9$.

As $a_7 | \ell_6 | a_9, a_8 | \ell_7 | a_{10}, a_9 | \ell_8 | a_{11},$ we see that $y(a_7)>y(a_9)>y(a_{11}), y(a_8)>y(a_{10})>y(a_{11})$, so we eliminate $a_7,a_8,a_9$. Similarly, as $b_2 | \ell_2 | b_4, b_3 | \ell_3 | b_5, b_4 | \ell_4 | b_6, b_5 | \ell_5 | b_7,$ so $y(b_2)<y(b_4), y(b_3)<y(b_5), y(b_6)<y(b_4), y(b_7)<y(b_5)$. So we eliminate $b_2,b_3,b_6,b_7$.

Finally, one can use the $\frac{\pi}{12}$ near squareness to estimate and get the result $y(a_1)>y(a_3)$, and so we can eliminate $a_1$.
\end{proof}

\begin{Prop}
If the condition for Prop ~\ref{Prop:3.5.1} is satisfied, and the quadrilateral is $\dfrac{\pi}{56}$-near square and the angle $\gamma$ is obtuse, then $\displaystyle \max y(b_1,b_4,b_5,b_8)=\max y(b_1,b_8), \min y(a_3,a_6,a_{11})=\min y(a_3,a_6)$
\label{Prop:3.5.3}
\end{Prop}
\begin{proof}
In this proof, we will show that $y(a_{11})>y(a_6), y(b_4)<y(b_1), y(b_5)<y(b_1)$.

Suppose the length $b_{10} b_{11} = 1$. Then by lemma ~\ref{Lem:Adjacent Edge Bound} we know:
\begin{equation*}
y(a_{11})\geq y(b_{11})+ \tan (\dfrac{\pi}{4}-\dfrac{\pi}{56}) \cos \dfrac{\pi}{28}\approx 0.8880356
\end{equation*}
We can also calculate that:
\begin{equation*}
y(a_6) = y(b_{11}) + \cos \alpha - \cos (\alpha+2\beta)+\cos (3\alpha+2\beta)-\cos(3\alpha+4\beta).
\end{equation*}
Note that $\frac{\pi}{2}-2\varepsilon<\alpha < \frac{\pi}{2}, \beta \geq \frac{\pi}{2}$. So, we have (by letting $\varepsilon=\frac{\pi}{56}$)
\begin{align*}
y(a_6)&\leq y(b_{11}) +  \cos \left(\frac{\pi}{2}-2\varepsilon\right) - \cos \left(\frac{\pi}{2}-2\varepsilon+2\cdot\frac{\pi}{2}\right)+\cos \left(3\left(\frac{\pi}{2}-2\varepsilon\right)+2\cdot\frac{\pi}{2}\right)
\\ &-\cos\left(3\left(\frac{\pi}{2}-2\varepsilon\right)+4\cdot\frac{\pi}{2}\right)\approx 0.884487
\end{align*}
A calculation shows that $y(a_{11})>y(a_6)$.

We have shown that $a_{11}$ is the lowest of $a_7, a_8, a_9$ and $a_{10}$. If we draw a line through $a_6$ perpendicular to $\ell_9$, then $a_7, a_8, a_9, a_{10}, a_{11}$ would all lie above this line. So for the same reason, if we draw a line through $b_1$ perpendicular to $\ell_5$, then $b_2, b_3, b_4, b_5, b_6$ would all lie below this line. As $\ell_5$ has positive slope, clearly this perpendicular line has negative slope. So we see that $y(b_4)<y(b_1)$ and $y(b_5)<y(b_1)$.
\end{proof}

By combining the Propositions and Remark, we conclude that a quadrilateral has orbit $X_2$ if the following conditions are satisfied: 
\begin{enumerate}
\item the quadrilateral is $\dfrac{\pi}{56}$-near square;
\item the quadrilateral has one acute angle $\alpha$, and two obtuse angle adjacent to $\alpha$;
\item one of the obtuse angle $\beta$ adjacent to $\alpha$ satisfies the inequalities $2\beta>\pi$ and $\alpha+3\beta<2\pi$.
\end{enumerate}
Note that the region $X_2$ in Figure ~\ref{fig:choplefthalf} does satisfy these properties, and hence have this periodic billiard path. (It is also worth mentioning that we do not need $\beta$ to be the smaller of the two angles adjacent to $\alpha$.)

\begin{figure}[h]
\includegraphics[width=\textwidth]{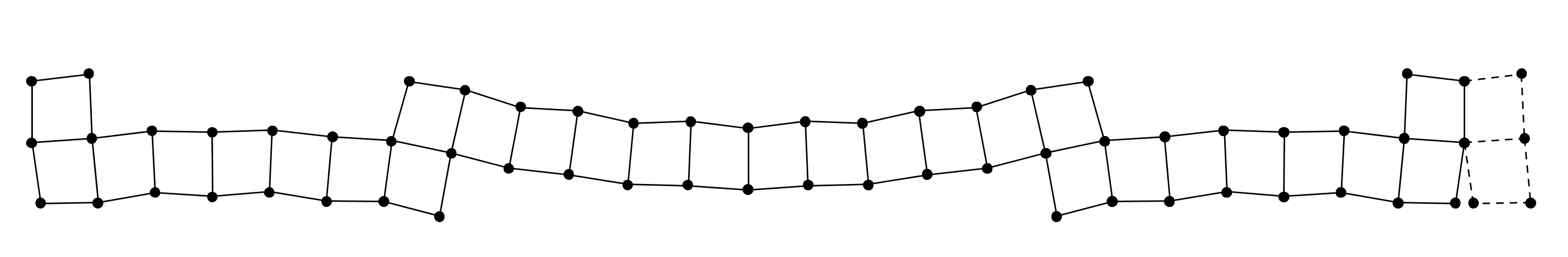}
\caption{Unfolding of $X_3$ orbit}
\label{fig:X3}
\end{figure}

\begin{figure}[h]
\includegraphics[width=\textwidth]{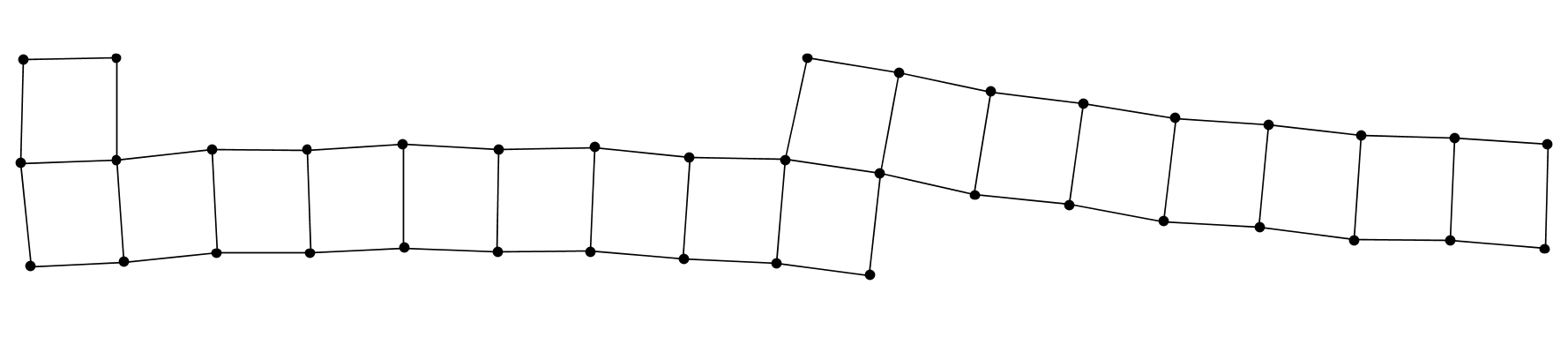}
\caption{Half unfolding of $X_4$ orbit}
\label{fig:X4half}
\end{figure}

Now we are ready to explore other orbits in the $X$ family. In general, the orbit $X_n$ will have orbit type $0(13)^n0(13)^{n-1}131(31)^{n-1} 0(31)^n03$, and some unfolding are in Figure ~\ref{fig:X3}, ~\ref{fig:X4half}. The conditions for these orbits to work is similar to $X_2$, and the proof is basically identical. For any integer n, a quadrilateral has orbit $X_n$ if the following conditions are satisfied:

\begin{enumerate}
\item the quadrilateral is $\varepsilon_n$-near square;

\item the quadrilateral has one acute angle $\alpha$, and two obtuse angle adjacent to $\alpha$;

\item one of the obtuse angle $\beta$ adjacent to $\alpha$ satisfies the inequalities $(n-2)\alpha+n\beta>(n-1)\pi$ and $(n-1)\alpha+(n+1)\beta<n\pi$.
\end{enumerate}

Here $\varepsilon_n$ is a decreasing sequence with limit $0$ as $n$ goes to $\infty$. $\varepsilon_n$ the real root nearest to $0$ of the following equation: Let $x_n=\dfrac{\pi}{2}-2\epsilon_n, y_n=\dfrac{n-1}{n}\pi-\dfrac{n-2}{n}x_n, y_n'=\frac{\pi}{2}$, then:

\begin{equation}
\tan\left(\dfrac{\pi}{4}-\varepsilon_n\right)\cos(2\varepsilon_n)=\sum_{i=1}^n \cos\left[(2i-1)x_n+(2i-2)y_n\right]-\sum_{i=1}^n \cos\left[(2i-1)x_n+(2i)y_n'\right]
\end{equation}

\begin{wrapfigure}{r}{6cm}
\vspace{-20pt}
  \begin{center}
    \includegraphics[width=6cm]{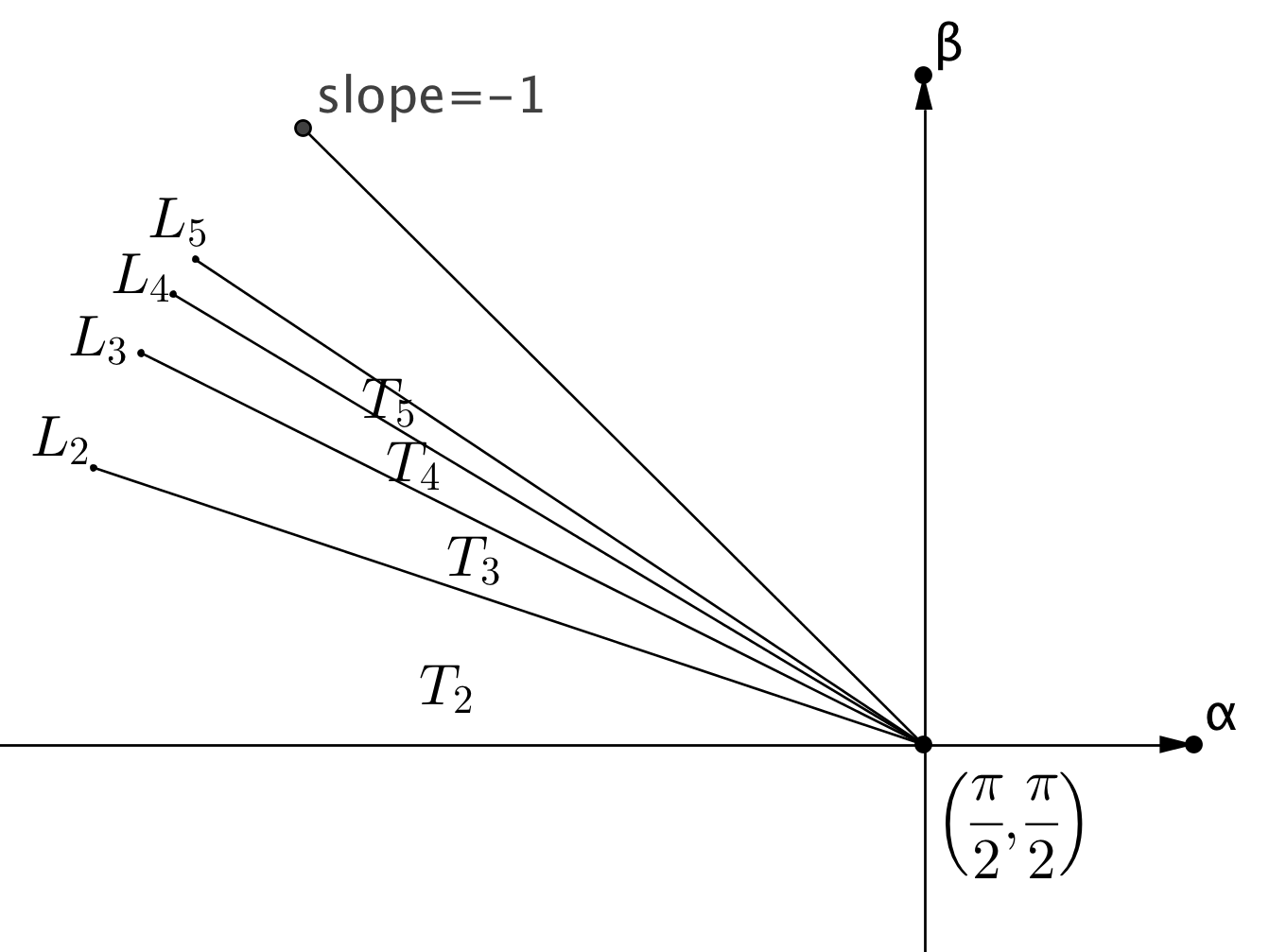}
  \end{center}
  \caption{Coordinate Space}
\label{fig:Coordinate Space 2}
\vspace{-10pt}
\end{wrapfigure}

Condition 2 and 3 implies that $X_n$ covers the region $T_n$ in the $\alpha$-$\beta$ plane of Figure ~\ref{fig:Coordinate Space 2} (as long as the quadrilateral is sufficiently near square). The gradient of $L_n$ is $-\frac{n-1}{n+1}$, which goes to $-1$ as $n$ goes to $\infty$. Note that gradient of $L_3$ is $-1/2$, and the orbit $A$ in Section 3.4 covers the region above the line $L_3$. 

In short, if the quadrilateral is described in region $T_3$ [$T_2$, respectively] and is $\varepsilon_3$-near square [$\varepsilon_2$], then it has the orbit $X_3$ [$X_2$].

One can calculate that $\varepsilon_3$ is slightly larger than $\frac{\pi}{107}$, and so the near-squareness condition for $X_3$ is more strict than that of $X_2$. The only thing left is the lines $L_2, L_3$, which will be solved by the Y family.

\subsection{The $Y$ family orbits}
\label{sec:Y}
\begin{figure}[h]
\hspace{-30pt}
\centering
\subfloat[]{\includegraphics[width=5.5cm]{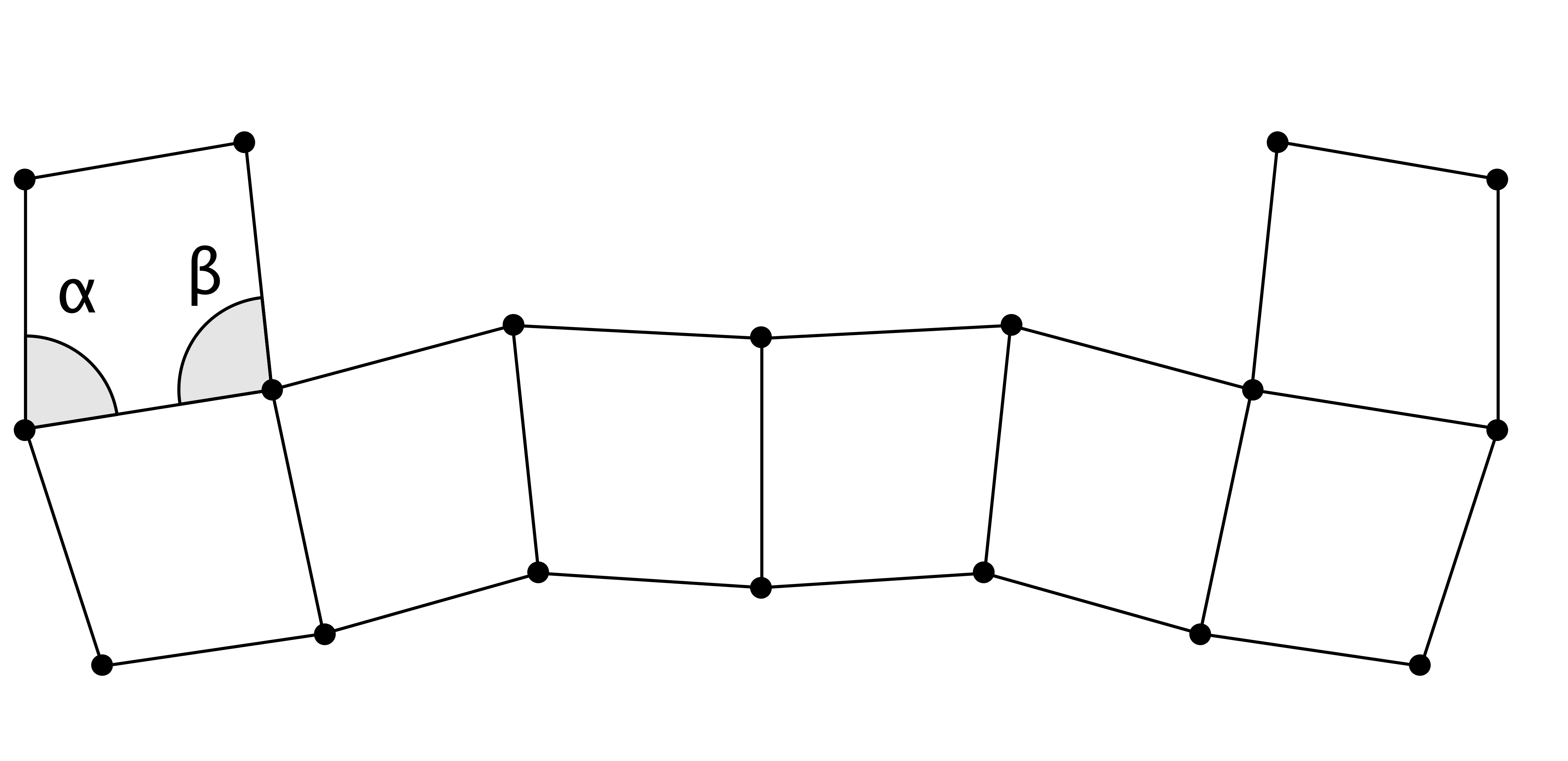}\label{fig:Y2}}
\subfloat[]{\includegraphics[width=7.5cm]{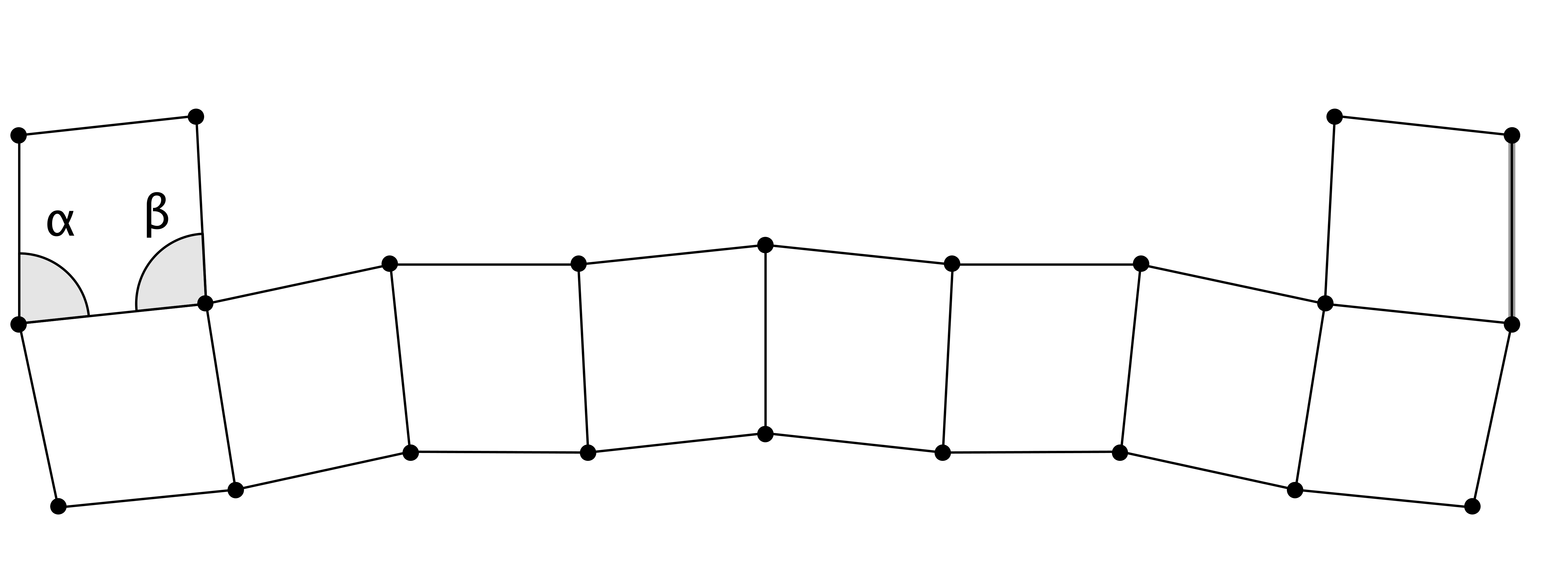}\label{fig:Y3}}
\caption{Unfolding for (a)$Y_2$, (b) $Y_3$}
\end{figure}

The $Y$ family orbits is a family of unstable orbit which covers exactly $L_n$ in Figure ~\ref{fig:Coordinate Space 2}. Some unfoldings are shown in Figure ~\ref{fig:Y2}, ~\ref{fig:Y3} with the parameter angle $\alpha$ and $\beta$ marked.

For example, let us look at $X_2$ in Figure ~\ref{fig:X2}. If the quadrilateral lies on $L_2$ in the $\alpha$-$\beta$ parameter plane, then the line $\ell_4$ is vertical. So, if we reflect the first four square in the unfolding of Figure \ref{fig:X2}, we would get the unfolding in Figure ~\ref{fig:Y2}, which we call $Y_2$. The whole $Y$ family arises this way, and the near-square condition (for this orbit to work) is looser than that for the $X$ family (since there are less vertices to consider compared to $X$ orbits).

To sum up, if a quadrilateral is $\varepsilon_n$-near square and lies on the line $L_n$ in the $\alpha-\beta$ parameter plane, then it has orbit $Y_n$.


\subsection{Summary}

So now let us put the whole proof together. Any quadrilateral $q$, suppose it is $\dfrac{\pi}{107}$-near square. 

\begin{enumerate}
\item If it is a rectangle, then it has a periodic path where the ball bounces between two parallel lines.

\item If $q$ is not a rectangle, then it has at least one acute angle, call it $\alpha$. If one of the angle adjacent to $\alpha$ is acute, then $q$ has orbit $F$. If one of the angle adjacent to $\alpha$ is $90^\circ$, then $q$ has orbit $R$. 

\item Suppose that both angles adjacent to $\alpha$ are obtuse, and let us call the smaller one $\beta$, the larger one $\gamma$. If the angle opposing $\alpha$ is acute, then $q$ has orbit $A$. 

\item Assume that the angle opposing $\alpha$ is larger than or equal to $\dfrac{\pi}{2}$ and call it $\theta$. Then we know $\alpha+2\beta+\dfrac{\pi}{2}\leq\alpha+\beta+\gamma+\theta=2\pi$, so $\alpha+2\beta\leq\dfrac{3\pi}{2}$. So in the $\alpha$-$\beta$ plane, the only region left is $T_2, T_3, L_2, L_3$. And these regions are covered by orbit $X_2, X_3, Y_2, Y_3$.
\end{enumerate}

This completes our proof for Theorem~\ref{Thrm:epsilon107}.

	 %
	 %

\section{Proof of Theorem ~\ref{Thrm:Rectangle Cover}}

The proof is pretty much identical to that of Theorem ~\ref{Thrm:epsilon107}. All these orbits still works for rectangles. Fix a rectangle $K$. For any quadrilateral $q$, suppose it is sufficiently near $K$. If $q$ is a rectangle, then clearly it has a periodic path. If it has two adjacent acute angle, then it has orbit $F$. If it has one acute angle adjacent to a right angle, then it has orbit $R$. If it has two acute angle opposing each other, then it will have orbit $A$.

Suppose it has at least one acute angle $\alpha$, two adjacent obtuse angle, and the opposite angle to $\alpha$ is not acute. Let $\beta$ be the smallest of the obtuse angles adjacent to $\alpha$. Again, we have the inequality $\alpha+2\beta<\dfrac{3\pi}{2}$, and the orbits $X_2, X_3, Y_2, Y_3$ covers these remaining regions in the $\alpha$-$\beta$ parameter plane. 

The critical question is how near to the fixed rectangle $K$ do we need? Unfortunately this is not easy to answer, as this condition depends on which fixed rectangle $K$ we choose. The ``flatter'' the rectangle $K$ is, the stricter the condition will need to be.
\\[1.0cm]\textbf{Acknowledgement.} We would like to thank Prof. W Patrick Hooper for his mentoring and support throughout our research, and for providing us with `McBilliards II', a Java program that searches for periodic billiard paths. We drew lots of observation and inspiration from this Java program. `McBilliards II' is available online at \url{http://wphooper.com/visual/mcb2/}.

\end{document}